\documentclass[11pt,reqno]{article}

\usepackage{amsmath,amssymb,amsthm}
\usepackage{mathtools, mathrsfs}

\usepackage[margin=1in]{geometry}
\usepackage{indentfirst}
\usepackage{titlesec}
\usepackage{datetime,comment}
\usepackage[utf8]{inputenc}
\usepackage{enumitem}

\usepackage{xcolor}

\usepackage[colorlinks,linkcolor = blue,citecolor = blue]{hyperref}

\usepackage{cite}

\title{Uniqueness Criteria for the Oseen Vortex in the 3d Navier-Stokes Equations}
\author{Jacob Bedrossian\thanks{\footnotesize Department of Mathematics, University of Maryland, College Park, MD 20742, USA \href{mailto:jacob@math.umd.edu}{\texttt{jacob@math.umd.edu}}. J.B. was supported by NSF CAREER grant DMS-1552826} \and William Golding\thanks{\footnotesize Department of Mathematics, University of Texas, Austin, TX 78712, USA \href{mailto:wgolding@utexas.edu}{\texttt{wgolding@utexas.edu}}. W.G. was partially supported by NSF-DMS grant 1840314}}
\titleformat{\section}
{\normalfont\Large\bfseries}
{\thesection.}{0.5em}{}
\titlespacing{\section}{12pc}{1.5ex plus .1ex minus .2ex}{1pc}

\titleformat{\subsection}
{\normalfont\bfseries}
{\thesubsection.}{0.4em}{}

\newtheorem{theorem}{Theorem}[section]
\newtheorem{proposition}{Proposition}[section]
\newtheorem{corollary}{Corollary}[theorem]
\newtheorem{lemma}{Lemma}[section]
\theoremstyle{definition}
\newtheorem{definition}{Definition}[section]
\newtheorem{remark}{Remark}
\theoremstyle{definition}

\numberwithin{equation}{section}

\DeclareMathOperator{\T}{\mathbb{T}}
\DeclareMathOperator{\R}{\mathbb{R}}

\DeclareMathOperator{\Z}{\mathbb{Z}}
\DeclareMathOperator{\N}{\mathbb{N}}
\DeclareMathOperator{\A}{\mathcal{A}}
\DeclareMathOperator{\F}{\mathcal{F}}
\let\S\relax
\DeclareMathOperator{\S}{\mathcal{S}}

\DeclareMathOperator{\tens}{\otimes}

\DeclareMathOperator{\embeds}{\hookrightarrow}
\DeclareMathOperator{\weakstar}{\rightharpoonup^*}

\let\div\relax
\let\L\relax

\DeclareMathOperator{\L}{\mathcal{L}}
\DeclareMathOperator{\Tau}{\mathcal{T}}
\DeclareMathOperator{\div}{\mathrm{div}}
\DeclareMathOperator{\sdv}{\overline{\nabla}}
\DeclareMathOperator{\slap}{\overline{\Delta}}
\DeclareMathOperator{\sdz}{\overline{\partial_{\mathit{z}}}}
\DeclareMathOperator{\divz}{\overline{\div}}

\DeclareMathOperator{\grad}{\nabla}
\DeclareMathOperator{\Real}{\mathbb{R}}

\newcommand{\norm}[1]{\left|\left| #1 \right|\right|}
\newcommand{\abs}[1]{\left| #1 \right|}
\newcommand{\set}[1]{\left\{ #1 \right\}}
\newcommand{\brak}[1]{\left\langle #1 \right\rangle}

\begin{document}
\maketitle
\abstract{
  In this paper, we consider the uniqueness of solutions to the 3d Navier-Stokes equations with initial vorticity given by $\omega_0 = \alpha e_z \delta_{x = y = 0}$, where $\delta_{x=y= 0}$ is the one dimensional Hausdorff measure of an infinite, vertical line and $\alpha \in \mathbb R$ is an arbitrary circulation.
  This initial data corresponds to an idealized, infinite vortex filament. One smooth, mild solution is given by the self-similar Oseen vortex column, which coincides with the heat evolution. Previous work by Germain, Harrop-Griffiths, and the first author implies that this solution is unique within a class of mild solutions that converge to the Oseen vortex in suitable self-similar weighted spaces.  In this paper, the uniqueness class of the Oseen vortex is expanded to include any solution that converges to the initial data in a sufficiently strong sense. This gives further evidence in support of the expectation that the Oseen vortex is the only possible mild solution that is identifiable as a vortex filament. The proof is a 3d variation of a 2d compactness/rigidity argument in $t \searrow 0$ originally due to Gallagher and Gallay.   
}

\tableofcontents


\section{Introduction}
We consider the incompressible Navier-Stokes equations on $\R^3$ written in vorticity form as
\begin{equation}
\begin{aligned}
& \partial_t \omega + u \cdot \nabla \omega - \omega \cdot \nabla u = \Delta \omega\label{vorticity}\\
& \nabla \cdot \omega = 0
\end{aligned}
\end{equation}
where $\omega = \grad \times u$ is the vorticity, and $u$ is the velocity, which is determined from $\omega$ via the Biot-Savart law:
\begin{equation}
u = -\nabla \times \Delta^{-1} \omega.
\end{equation}
We are interested in measure-valued initial vorticity concentrated along a smooth curve in $\R^3$. 
That is, $\omega_0 = \alpha\delta_\Gamma$, where $\alpha\in\R$ is an arbitrary circulation number and $\Gamma \subset \R^3$ is a smooth, oriented curve, and $\delta_\Gamma$ denotes the vector-valued Radon measure defined via its action on continuous compactly supported test functions, $\varphi\in C_c(\R^3;\R^3)$, as
\begin{equation}\label{vortexfilament}
\langle \varphi, \delta_{\Gamma} \rangle = \int_{\Gamma} \varphi \cdot ds,
\end{equation}
where $ds$ is the one-dimensional surface measure on $\Gamma$.
This initial data is not $L^2_{loc}$, so the theory of Leray-Hopf weak solutions does not apply.
Instead, the framework of mild solutions is most natural, i.e. solutions of the corresponding integral equation,
\begin{equation}\label{mild}
\omega(t) = e^{t\Delta}\omega_0 - \int_0^t e^{(t-s)\Delta}\nabla \cdot \left[u(s) \tens \omega(s) - \omega(s) \tens u(s)\right] \ ds.
\end{equation}
We call solutions to (\ref{mild}) with initial data of type $\omega_0 = \alpha \delta_{\Gamma}$ \emph{vortex filament solutions}. 
The prototypical vortex filament solution is the self-similar \emph{Oseen vortex column}.  
Consider the infinite straight filament, $\omega_0$ concentrated along the $z$-axis. Decomposing $\R^3$ as $\R^2 \times \R$ with coordinates $(x,z)$, we write $\omega_0 = \alpha \delta_{x = 0}$. The Oseen vortex is given by the heat evolution: defining
\begin{equation}
G(\xi) = \frac{1}{4\pi}e^{-\frac{|\xi|^2}{4}} \qquad \text{and} \qquad v^G(\xi) = \frac{1}{2\pi}\frac{\xi^\perp}{|\xi|^2}\left(1 - e^{-\frac{|\xi|^2}{4}}\right), 
\end{equation}
we have
\begin{equation}\label{explicit}
\omega^g(t,x,z) = \frac{\alpha}{t}\begin{bmatrix} 0\\  0 \\ G\left(\frac{x}{\sqrt{t}}\right)\end{bmatrix} \qquad \text{and}\qquad u^g(t,x,z) = \frac{\alpha}{\sqrt{t}}\begin{bmatrix} v^G\left(\frac{x}{\sqrt{t}}\right)\\ 0 \end{bmatrix}.
\end{equation}
Note that these are simply the 3d extensions of the point-vortex solutions of the 2d vorticity equations, whose uniqueness properties were studied in \cite{gallaghergallaylions} (see also \cite{gallaghergallay}).  

A great deal of interest in such initial data stems from experimental observations that vorticity tends to concentrate on 1-dimensional structures in $\R^3$ known as \emph{vortex filaments}.  Vortex filaments are observed to propagate coherently on relatively large time scales, for example, as in wing-tip vortices (see e.g. \cite{DevenportEtAl}). Moreover, vortex filaments are thought to potentially play a role in intermittency in turbulent flows (see e.g. \cite{MoffatEtal94}).
After early work by Helmholtz and Kelvin \cite{helmholtz,kelvin1,kelvin2}, da Rios formally derived the binormal approximation for the evolution of smooth, non-self-intersecting vortex filaments \cite{daRios}. This approximation asserts that the center-line of a vortex filament in the 3d Euler equations tightly concentrated to be width $\delta$ will (approximately) evolve with a velocity given by $V = \frac{1}{4\pi} \alpha \kappa \log \delta \mathbf b$ where $\alpha$ is the circulation, $\kappa$ is the curvature, and $\mathbf{b}$ is the binormal.
This approximation was subsequently re-discovered several times in the physics literature and is now usually called the local induction approximation (LIA) (see \cite{JS17,MajdaBertozzi,Ricca96} and the references therein). In the mathematics literature, the work of \cite{JS17} justifies the LIA for solutions of the 3d Euler equations assuming that the vorticity remains concentrated in a sufficiently small tube.
The limit $\delta \to 0$ is a very singular limit for the Euler equations and hence it is natural also to study vortex filaments in the Navier-Stokes equations, allowing the viscosity to provide a natural regularization (as opposed to a finite width assumption).
It is also worth remarking that vortex filaments may be subject to high frequency instabilities at the scale of the vortex core (see e.g. \cite{WidnallEtal1974}), posing an additional difficulty in justifying any $\delta \to 0$ asymptotics. Finally, we remark that the binormal flow itself has been intensively studied in the literature (e.g.~\cite{MR3363420,MR3353807,MR3059243,MR3429472,MR3116002,MR3204898,MR2862038,MR2472037,2018arXiv180706948B}). 

Another part of the interest in such initial data is that it is in some sense, exactly on the borderline of the local well-posedness theory of Navier-Stokes. 
To our knowledge, initial data of the type \eqref{vortexfilament} for the Navier-Stokes equations was first studied by Giga and Miyakawa \cite{gigamiyakawa}, who proved global well-posedness assuming $\abs{\alpha}$ sufficiently small. Axisymmetric vortex filaments (i.e. vortex rings) were studied more recently in \cite{FengSverak15,GallaySverak16} providing existence and uniqueness in the axisymmetric class for $\alpha$ arbitrary. 
The recent work \cite{jacob} provides the first existence and (limited) uniqueness results for vortex filament solutions with arbitrary $\alpha$ and no symmetry assumptions (as well as sufficiently regular background vorticity and multiple filaments). 
As discussed in \cite{gigamiyakawa},  the initial data (\ref{vortexfilament}) lives in a Morrey space of vector-valued measures, $M^{3/2}$, the set of Radon measures such that 
\begin{equation}
\|\mu\|_{M^{3/2}} = \sup_{r > 0, x\in\R^3}\frac{ |\mu(B_r(x))|}{|r|} < \infty. 
\end{equation}
If $\omega$ is a solution to (\ref{vorticity}), then for any $\lambda > 0$, $\omega_\lambda$ is a solution to (\ref{vorticity}), where
\begin{equation}
\omega_\lambda(t,x) = \frac{1}{\lambda}\omega\left(\frac{t}{\lambda},\frac{x}{\sqrt{\lambda}}\right) \qquad \text{and} \qquad u_\lambda(t,x) = \frac{1}{\sqrt{\lambda}}u\left(\frac{t}{\lambda},\frac{x}{\sqrt{\lambda}}\right),\label{scaling}
\end{equation}
and hence we see that $M^{3/2}$ is scale-invariant; recall that spaces of this type are called \emph{critical} (the critical Lebesgue space in vorticity form is $L^{3/2} \subset M^{3/2}$). Although such spaces are natural borderline spaces, critical spaces need not preserve the relative size of linear and non-linear effects.
In velocity form, the most well-known critical spaces are $\dot{H}^{1/2} \subset L^3 \subset L^{3,\infty} \subset BMO^{-1} \subset B^{-1}_{\infty,\infty}$. 
In \cite{fujitakato,kato}, local-in-time well-posedness was proved for $\dot{H}^{1/2}$ and $L^3$ respectively for all initial data. 
In such spaces one can take time small to make the nonlinear contributions smaller and hence treat the equation as a perturbation of the heat equation.  
Critical spaces which contain self-similar solutions, such as $L^{3,\infty}$ in velocity form and $M^{3/2}$ in vorticity form, clearly cannot satisfy this property, and so are in some sense ``more critical''.  We follow \cite{jacob} and call a space $X$ \emph{ultra-critical} if $X$ is critical and the Schwartz class is \emph{not} dense in $X$; spaces containing the initial data of self-similar solutions are of this type. 
In 3d, global well-posedness for small data is known in $L^{3,\infty}$ and $BMO^{-1}$ \cite{kochtataru} in velocity form and in $M^{3/2}$ in vorticity form \cite{gigamiyakawa}; see also \cite{chemin,cannone1,cannoneplanchon,planchon,cannone2,Taylor92,HideoMasao94} for numerous refinements). 
However, even existence of large\footnote{by which we mean a large part of the initial data cannot be approximated by Schwartz class functions.} data mild solutions is open except for a small set of trivial examples, the self-similar solutions constructed in $L^{3,\infty}$ \cite{haosverak}, the $M^{3/2}$ vortex rings studied in \cite{FengSverak15,GallaySverak16}, and the $M^{3/2}$ vortex filaments solutions constructed in \cite{jacob}. Local-in-time ill-posedness even for small initial data in $B^{-1}_{\infty,\infty}$ was proved in \cite{bourgain}.

In 2d, the ultra-critical class for the vorticity equations analogous to $M^{3/2}$ is the space of finite Borel measures taken with the total variation norm. Existence was proved in \cite{cottet} and \cite{gigamiyakawaosada} (independently). Well-posedness for small atomic parts was also proved in \cite{gigamiyakawaosada}.
In \cite{gallaghergallaylions}, it was proved that the Oseen vortex is the unique mild solution in 2d with initial data given by $\alpha \delta_{x=0}$. This result plays an important role underlying the well-posedness for arbitrary measures ultimately proved by Gallagher and Gallay in \cite{gallaghergallay} (see also \cite{Gallay12,GallagherPlanchon02,Pierre06}). 

As discussed in \cite{JiaSverak15,GuilllodSverak17}, local well-posedness for 3d Navier-Stokes in ultra-critical regularity is not expected in general.
For example, under natural spectral conditions (still unverified for any concrete examples), instabilities will dominate and lead to multiple smooth solutions from the same locally scale-invariant initial data \cite{JiaSverak15}.
Despite the potential difficulties, \cite{jacob} demonstrates a well-posedness class in the neighborhood of self-similar, and (in some sense) approximately self-similar initial data, provided the filament is smooth and non-self-intersecting (and it is quite possible that well-posedness is false otherwise, e.g. if the filament is not sufficiently smooth or there are intersections).
Roughly speaking, it is proved that there is only one solution that looks like the Oseen vortex column if one zooms into the singularity at $t \searrow 0$. Moreover, heuristically at least, \cite{jacob} demonstrates that curvature effects are `sub-critical' and that the true, `ultra-critical' difficulty is in understanding perturbations of the Oseen vortex column itself.

\section{Main results}
In this paper, we improve the uniqueness criteria for the Oseen vortex solution \eqref{explicit} provided in \cite{jacob}, that is, we expand the class of mild solutions for which we can deduce \eqref{explicit} is the only possible solution with initial data $\omega_0 = \alpha \delta_{x=0}$. See Remark \ref{rmk:Meanings} for a discussion of the meaning of our class. 
We will be concerned with two cases, the infinite straight filament and the straight filament with periodic boundary conditions. In order to define notation for both cases simultaneously, in the following, let $D = \T, \R$ and $D^* = \Z,\R$, the corresponding Pontryagin dual group.

\subsection{Preliminaries}
We recall the (partial) self-similar coordinates used in \cite{jacob}. 
For $x\in \R^2$, $t\in \R^+$, and $z\in D$, define $\xi \in \R^2$ and $\tau \in \R$ via
\begin{equation}
(\tau,\xi,z) = \left(\log(t),\frac{x}{\sqrt{t}},z\right)
\end{equation}
together with the corresponding rescaling of the vorticity and velocity as
\begin{equation}
w(\tau,\xi,z) = e^\tau \omega(e^\tau,e^{\tau/2}\xi,z) \qquad v(\tau,\xi,z) = e^{\tau/2}u(e^\tau,e^{\tau/2}\xi,z).\label{rescaling}
\end{equation}
For the remainder of the paper, $(\omega,u)$ will denote a solution pair in standard coordinates, whereas $(w,v)$ will denote a solution pair in self-similar coordinates, i.e. related to $(\omega,u)$ via (\ref{rescaling}). 
In self-similar coordinates, \eqref{vorticity} becomes 
\begin{equation}\label{svorticity}
\partial_\tau w + (v\cdot \sdv)w - (w \cdot \sdv)v = \L + \sdz^2 w, \qquad v = (-\slap)^{-1}\sdv\times w.
\end{equation}
where 
\begin{equation}
\sdv = \begin{bmatrix} \nabla_\xi\\  \sdz\end{bmatrix} = \begin{bmatrix} \partial_{\xi_1}\\ \partial_{\xi_2} \\ e^{\tau/2}\partial_z\end{bmatrix}, \quad \slap = \Delta_\xi + \sdz^2 = \Delta_\xi + e^{\tau} \partial_{zz},\quad \overline{\textup{div}} = \sdv \cdot, 
\end{equation}
and $\L$ is the 2D Fokker-Planck operator
\begin{equation}
\L = \Delta_{\xi} + \frac{1}{2} \xi \cdot \nabla_\xi  + 1.\label{FokkerPlanck}
\end{equation}
Since our differential operators are time dependent, there might occasionally be ambiguity within an equation as to at what time a derivative is taken. In such instances, we will use the notation $\sdv_\tau$ or $\slap_\tau$ to denote the corresponding operator applied at time $\tau$. 

Define the weighted spaces $L^p(m)$ for $m \in \R^+$ via the norm  
\begin{equation}
\|f\|_{L^p(m)}^p = \|\langle \xi \rangle^m f(\xi)\|_{L^p_\xi}^p = \int_{\R^2} \langle \xi \rangle^{pm} f(\xi)^p \ d\xi\label{lpm norm},
\end{equation}
where $\brak{\xi} = (1 + |\xi|^2)^{1/2}$. 
The Fokker-Planck operator and the 2d Navier-Stokes equations linearized around the Oseen vortex both have spectral gaps in $L^2(m)$ for suitable values $m$, which plays a major role in the works of Gallay and Wayne \cite{gallaywayne1,gallaywayne2}, Gallagher and Gallay \cite{gallaghergallay}, and \cite{jacob}.


We follow \cite{jacob} and use the Wiener algebra in $z$ to build our spaces. For $D$ the $z$ domain, and a Banach space $X$ of measurable functions $g:\R^2\rightarrow \R$, we define $B_zX$ as the associated $X$-valued Wiener algebra of measurable functions $f:\R^2\times D \rightarrow \R$ with the norm,
\begin{equation}
\|f\|_{B_zX} = \int_{D^*} \|\hat f(\zeta)\|_X \ d\zeta,
\end{equation} 
where the Fourier transform is taken only in the $z$-direction and $d\zeta$ denotes the Haar measure corresponding to $D^*$ (either Lebesgue measure or the counting measure).

Finally, we state more precisely the definition of mild solution. We denote $\omega \in C_w([0,T];M^{3/2})$ if for all $t \in [0,T]$, $\omega(t)$ is in $M^{3/2}$,  $\sup_{t \in [0,T]} \norm{\omega(t)}_{M^{3/2}} < \infty$, and $\omega$ is continuous in the weak-$\star$ topology of measures. 
\begin{definition}[Mild Solution] \label{def:Mild}
We say that $\omega \in C_w([0,T];M^{3/2})$ is a \emph{mild solution} of the 3d Navier-Stokes equations with initial data $\mu \in M^{3/2}$ if:
\begin{itemize}
\item for $[0,T]$, $\omega(t)$ satisfies the vorticity equation in the mild sense, i.e. $\omega$ satisfies (\ref{mild}) and the Bochner integral converges in $C_w([0,T];M^{3/2})$;
\item $\omega(t) \rightharpoonup^\ast \mu$ as $t \searrow 0$. 
\end{itemize}
\end{definition}

\subsection{Statement of Main Results}
We are now in a position to state our first main result, which pertains to periodic-in-$z$ solutions. 
The results of \cite{jacob} imply that if the solution in self-similar variables $w(\tau)$ satisfies $w(\tau) \to e_z \alpha G$ in $B_z L^2(m)$ as $\tau \to -\infty$, then $w \equiv e_z \alpha G$ for all $\tau$. 
Here, we improve this uniqueness result to assert that if $w(\tau)$ asymptotically as $\tau \rightarrow -\infty$ is sufficiently two-dimensional, then it must be the Oseen vortex (\ref{explicit}). In the original variables, this can be understood as assuming that the initial data is obtained in a sufficiently strong sense. 

\begin{theorem}\label{thm1}
Suppose $\omega \in C_w([0,T];M^{3/2})$ is a mild solution in the sense of Definition \ref{def:Mild} defined on $\R^2 \times \T$ with initial data $\alpha \delta_{x=0}$ for some $\alpha \in \R$ and that $w$ is the corresponding self-similar version of $\omega$ defined via (\ref{rescaling}).
 If for some $m > 2$,
\begin{itemize}
\item[(i)]  $\limsup_{\tau \to -\infty}\norm{ w(\tau)}_{B_z L^2(m)} < \infty$;
\item[(ii)] $\lim_{\tau \to -\infty}\norm{w^\xi(\tau)}_{B_z L^2(m)} = 0$;
\item[(iii)]  and $\limsup_{\tau \to -\infty}\norm{\partial_z w^z(\tau)}_{B_z L^2(m)} < \infty$,
\end{itemize}
then $\omega$ is the Oseen vortex defined in (\ref{explicit}), i.e. $\omega \equiv \omega^g$ for all $(0,\infty)\times \R^2\times \T$. 
\end{theorem}


\begin{remark} \label{rmk:Meanings}
Assumptions (i)--(iii) can be interpreted as a strengthening of $\omega(t) \rightharpoonup^\ast \alpha \delta_{x = 0}$ as $t \searrow 0$. Assumption (i) controls the magnitude of the singularity and provides the natural, quantitative control on the spreading of the solution from where the initial data is concentrated. Assumption (ii) provides additional control on the potential oscillations of $\omega^\xi$ as $\omega^\xi \rightharpoonup^\ast 0$. Assumption (iii) provides additional control on the potential oscillations of $\omega^z$ along the filament.     
\end{remark}

\begin{remark} \label{rmk:Aiii}
Assumption (iii) is taken to provide compactness as $\tau \to -\infty$ that is not automatic from parabolic regularity combined with Assumption (i), and can be replaced with the following compactness criterion: for all $\epsilon > 0$, $\exists R$ such that
\begin{align*}
\limsup_{\tau \to -\infty}\int_{\abs{\zeta} > R}\norm{\widehat{w^z}(\tau,\zeta)}_{L^2(m)} d\zeta < \epsilon. 
\end{align*}
This compactness is crucial for our proof, because parabolic regularity is only sufficient to prove $\|w(\tau)\|_{B_zL^2(m)} \lesssim 1$ implies $\|e^{\tau/2}\partial_z w^z(\tau)\|_{B_zL^2(m)} \lesssim 1$ (see Proposition \ref{prop5} in Section \ref{sec:parabolic}). 
\end{remark}

\begin{remark}
	The uniqueness result of \cite{jacob} only rules out other solutions which are very close to the Oseen vortex in the ultra-critical space $M^{3/2}$ as $t \to 0$. On the other hand, while Theorem \ref{thm1} does not imply the uniqueness result of \cite{jacob},
Theorem \ref{thm1} does rule out a wide variety of potential solutions far from the Oseen vortex in the ultra-critical space, at the cost
of requiring additional control on $z$-derivatives through Assumption (iii).
\end{remark}

From Theorem \ref{thm1}, we can deduce the following second criteria for uniqueness of $\R^2\times \T$, which is a slight variant.

\begin{corollary}\label{cor1}
  Suppose $\omega \in C_w([0,T];M^{3/2})$ is a mild solution in the sense of Definition \ref{def:Mild} defined on $\R^2 \times \T$ with initial data $\alpha \delta_{x=0}$ for some $\alpha \in \R$ and that $w$ is the corresponding self-similar version of $\omega$ defined via (\ref{rescaling}).
 If for some $m > 2$,
\begin{itemize}
\item[(i)]  $\limsup_{\tau \to -\infty}\norm{ w(\tau)}_{B_z L^2(m)} < \infty$;
\item[(ii)]  $\lim_{\tau\rightarrow -\infty}w^\xi(\tau) = 0$ in $B_zL^2(m)$ with the weak topology;
\item[(iii)]  and $\limsup_{\tau \to -\infty}\norm{\partial_z w(\tau)}_{B_z L^2(m)} < \infty$,
\end{itemize}
then $\omega$ is the Oseen vortex defined in (\ref{explicit}), i.e. $\omega = \omega^g$ pointwise in $(0,\infty)\times \R^2\times \T$.
\end{corollary}

\begin{remark}
Here, we note that we have made slightly stronger assumptions on the $z$-derivative, namely, we assume the full partial derivative $\partial_z w$ is bounded in $B_zL^2(m)$. However, we need only assume that the $\xi$-component of vorticity is weakly converging in this case.
\end{remark}

\begin{remark}
	We believe that Theorem \ref{thm1} and Corollary \ref{cor1} can be extended to provide a uniqueness class for general vortex filaments as studied in \cite{jacob} by applying methods used therein. However, this would require a much more involved analysis to carry out in detail. 
\end{remark}

The next result pertains to the infinite, straight filament in $\R^3$. In this case, we also require localization of the error in $z$. 

\begin{theorem}\label{thm2}
  Suppose $\omega \in C_w([0,T];M^{3/2})$ is a mild solution in the sense of Definition \ref{def:Mild} defined on $\R^2 \times \Real$ with initial data $\alpha \delta_{x=0}$ for some $\alpha \in \R$ and that $w$ is the corresponding self-similar version of $\omega$ defined via (\ref{rescaling}).
Set $w_c^\xi = w^\xi$ and $w_c^z = w^z - \alpha G$.
If for some $m > 2$,
\begin{itemize}
\item[(i)]  $\limsup_{\tau \to -\infty}\|w_c(\tau)\|_{B_zL^2(m)} < \infty$;
\item[(ii)]  $\lim_{\tau\to -\infty} \|w^\xi_c(\tau)\|_{B_zL^2(m)} = 0$;
\item[(iii)]  $\limsup_{\tau \to -\infty} \|\partial_z w^z_c(\tau)\|_{B_zL^2(m)} < \infty$;
\item[(iv)] and $\limsup_{\tau \to -\infty} \|z w^z_c(\tau)\|_{B_zL^2(m)} < \infty$,
\end{itemize}
then $\omega$ is the Oseen vortex defined in (\ref{explicit}), i.e. $\omega = \omega^g$ pointwise in $(0,\infty)\times \R^2\times \R$.
\end{theorem}

\begin{remark}
Note, that since $w^g := (0,0,\alpha G)^t$ is constant in $z$, the Fourier transform in $z$, $\hat{w}^g$, is a delta mass concentrated at frequency $\zeta = 0$ and therefore is not in $L^1_\zeta$. However, we can still say that perturbations of $w^g$ by elements $w_c$ of $B_zL^2(m)$ such that $w_c + w^g$ solves the vorticity equation must actually be the Gaussian $w^g$. Working with the perturbation, $w_c = w - w^g$, instead of $w$ will introduce some difficulties we will explain in Section 5.
\end{remark}

\begin{remark}
Analogous to Remark \ref{rmk:Aiii}, Assumptions (iii) and (iv) are only used to obtain compactness as $\tau \to -\infty$ by providing regularity that is not automatic from parabolic smoothing.
\end{remark}

\begin{remark}
Although Theorem \ref{thm2} is more difficult to prove, both Theorems \ref{thm1} and \ref{thm2} are of intrinsic physical interest as each rule out different types of potential pathological solutions. On one hand, Theorem \ref{thm1} rules out potential pathological solutions that are not localized in $z$, but do retain a discrete translation symmetry of the initial data. On the other hand, Theorem \ref{thm2} rules out potential pathological solutions that are localized in $z$ and break the translation symmetry of the initial data.
\end{remark}

\begin{remark} \label{rmk:JGHGvsJW}
As discussed in \cite{JiaSverak15,GuilllodSverak17}, the linearized stability of a self-similar solution is expected to play a crucial role in determining uniqueness or non-uniqueness of solutions nearby in a suitable sense. This requisite linearized stability was proved in \cite{jacob} for all $\alpha$ and plays a crucial role in the uniqueness result of \cite{jacob}: any solution satisfying $w(\tau) \to e_z \alpha G$ in $B_zL^2(m)$ as $\tau \to -\infty$ necessarily satisfies $w \equiv e_z \alpha G$.
In Theorem \ref{thm1}, the uniqueness result of \cite{jacob} is used essentially as a \emph{black box}. That is, Theorem \ref{thm1} is proved by a compactness/rigidity argument which shows that any solution satisfying Assumptions (i)--(iii) must satisfy $w(\tau) \to e_z \alpha G$ in $B_zL^2(m)$ as $\tau \to -\infty$ (and hence $w \equiv e_z \alpha G$ by \cite{jacob}).
The proof of Theorem \ref{thm2} additionally uses some of the linearized stability results contained in \cite{jacob} (as well as a few variations deduced below).  Finally, we remark that linearized stability of the Oseen vortex (and other column vortices) has also been recently studied in the inviscid case~\cite{MR4010658,gallay2018spectral}.
\end{remark}

\subsection{Conventions}

We make the following standard conventions:

\begin{itemize}
\item We define the Fourier transform as 
\begin{equation}
\hat{f}(\zeta) = \frac{1}{\sqrt{2\pi}}\int_{\R} e^{-i x \zeta} f(x) \ dx.
\end{equation}
\item We take the convention that constants that may change from line to line are denoted $C$ and $a \lesssim_\Lambda b$ if $a \le Cb$ where $C$ is a constant depending only on $\Lambda$.
\item Also, we denote by $|v|$ the standard Euclidean norm of $v\in \R^n$. 
\item We denote for $u,v\in \R^n$, we denote by $u \tens v$, the matrix $(u\tens v)_{ij} = u^iv^j$ and  $(\grad \cdot (u \otimes v))_i = \partial_j(u^iv^j) = (v \cdot \nabla)u^i + u^i(\nabla \cdot v)$.
\end{itemize}




\section{Preliminaries}

\subsection{Analysis in $B_zX$ Spaces}

Recall we defined the Banach spaces $B_zX$ via the norm
\begin{equation}
\|f\|_{B_zX} = \int_{D^*} \|\hat f(\zeta)\|_X \ d\zeta,
\end{equation}
where the Fourier transform is taken only in $z$ and $D$ is either $\R$ or $\T$. We introduce here a few fundamental inequalities and properties of the $B_zX$ spaces that we will use frequently.

\begin{lemma}[H\"older's Inequality (Lemma A.1; \cite{jacob})] \label{lem1}
For each $f,g\in B_zL_x^p, B_zL_x^q$, where $\frac{1}{r} = \frac{1}{p} + \frac{1}{q}$, then  $fg\in B_zL^r_x$ and
\begin{equation}\label{holder}
\|fg\|_{B_zL^r_x} \le \|f\|_{B_zL^p_x} \|g\|_{B_zL^q_x}.
\end{equation}
\end{lemma}

\begin{lemma}[Self-Similar Biot-Savart Law (Lemmas A.3 and A.4; \cite{jacob})]\label{lem2}
For $v,w$ related by the self-similar Biot-Savart law, that is, $v = -\sdv\times (\slap)^{-1}w$ which expands as
\begin{equation}\label{biotsavartstructure}
v = \begin{bmatrix} v^\xi\\ v^z\end{bmatrix} = \begin{bmatrix}
\sdz (-\slap)^{-1}(w^{\xi})^\perp - \nabla_\xi^\perp(-\slap)^{-1}w^z\\
\nabla_\xi^\perp \cdot (-\slap)^{-1} w^\xi
\end{bmatrix},
\end{equation}
satisfies
\begin{equation}
\|v\|_{B_zL^4} \lesssim \|w\|_{B_zL^{4/3}},\label{biotsavart}
\end{equation}
and for $1 < p < \infty$,
\begin{equation}
\|\sdv v\|_{B_zL^p} \lesssim \|w\|_{B_zL^p}.\label{riesz}
\end{equation}
\end{lemma}

\begin{lemma}\label{lem3}
The following embeddings hold for all $m > 1$ and $\epsilon > 0\mathrm{:}$
\begin{itemize}
\item $B_zX \embeds C_z(D;X);$
\item $B_zL^2(m) \embeds B_zL^1;$
\item $B_z\dot{H}^1(m) \embeds B_zL^4;$
\item $H^1(m+\epsilon) \embeds L^2(m)$;
\end{itemize}
where, moreover, the final embedding is compact.
\end{lemma}

\begin{proof}
The first embedding follows from the fact that the Fourier transform of an $L^1$ function is continuous. The second embedding follows from the fact that if $X\embeds Y$, then $B_zX\embeds B_zY$ and $L^2(m) \embeds L^1$ for $m > 1$ by H\"older's inequality. Similarly, the third embedding follows from the Sobolev embedding $\dot{W}^{1,4/3}(\R^2) \embeds L^4(\R^2)$ and the embedding $L^2(m) \embeds L^{4/3}$, 
\begin{equation}
\|f\|_{B_zL^4} \lesssim \|\nabla_\xi f\|_{B_zL^{4/3}} \lesssim_m \|\nabla_\xi f\|_{B_zL^2(m)}.
\end{equation}
Finally, for the compactness of the final embedding, take $\{f_n\}$ an arbitrary sequence in $L^2(m)$ with $\|f_n\|_{H^1(m+\epsilon)} \lesssim 1$. By Rellich-Kondrachov and a diagonalization argument, we may extract a subsequence $f_{n_k}$ and a limit $f$, such that for any compact $K\subset \R^2$, 
$$\lim_{k\rightarrow \infty} \|f - f_{n_k}\|_{L^2(K)} = 0.$$
Since $\|f_{n_k}\|_{L^2(m+\epsilon)} \lesssim 1$, we conclude that $f_{n_k} \rightarrow f$ in $L^2(m)$.
\end{proof}

\begin{lemma}\label{extensionlem}
Let $T:X \rightarrow Y$ be a bounded linear operator between two Banach spaces $X$ and $Y$. Then, $T$ extends to bounded linear map $\tilde{T}:B_zX \rightarrow B_zY$ such that $T$ is given by the extension formula,
\begin{equation}\label{extension}
(\tilde{T}f)(\xi,z) = T(f(\cdot,z))(\xi).
\end{equation}
Moreover, $\|\tilde T\|_{B_zX \rightarrow B_zY} = \|T\|_{X \rightarrow Y}$
\end{lemma}
\begin{proof}
Consider first the case of $D = \T$. The set of all functions
$$\mathcal{A}_N = \set{f = \sum_{\abs{k} < N} s_k e^{ikz} : s_k \in X}$$
is such that $\cup_N \mathcal{A}_N$  is dense in $B_zX$ and by linearity we have $\tilde T(f) = \sum_{\abs{k} < N} Ts_k e^{ikz}$ for all $f \in \mathcal{A}_N$.
Then, note that
\begin{align*}
\norm{\tilde{T} f}_{B_z Y} = \sum_{\abs{k} < N} \norm{T s_k}_{B_z Y} \leq \norm{T}_{X \to Y} \norm{f}_{B_z X},
\end{align*}
and so $\tilde{T}$ is a bounded linear operator on the dense set $\cup_N \mathcal{A}_N$, and so extends to a bounded linear operator $B_z X \to B_z Y$ by density with $\norm{\tilde{T}}_{B_z X \to B_z Y} \leq \norm{T}_{X \to Y}$. To see $\norm{\tilde{T}}_{B_z X \to B_z Y} \geq \norm{T}_{X \to Y}$ consider a sequence $\set{s_j}_{j=1}^\infty \in X$ with $\norm{s_j}_{X} = 1$ such that $\norm{Ts_j}_{B_z X} \to \norm{T}_{X \to Y}$ and consider the sequence $f_j = \frac{1}{2}s_j e^{iz} + \frac{1}{2}s_j e^{-iz}$.

Now consider the case $D = \R$. In this case, we use the dense set
\begin{align*}
\mathcal{A}_N = \set{f \in B_z X : \hat{f} = \sum_{1 \leq j \leq N} s_j \mathbf{1}_{A_j}(z), \quad A_j \in \mathcal{B}(\Real)}.  
\end{align*}
By density of simple functions in $L^1$, $\cup_N \mathcal{A}_N$ is dense in $B_z X$. The proof that $\norm{\tilde{T}}_{B_z X \to B_z Y} \leq \norm{T}_{X \to Y}$ proceeds from here in essentially the same manner as the case of $D = \T$.
To see  $\norm{\tilde{T}}_{B_z X \to B_z Y} \geq \norm{T}_{X \to Y}$ let $\set{s_j}_{j=1}^\infty \in X$ with $\norm{s_j}_{X} = 1$ such that $\norm{Ts_j}_{B_z X} \to \norm{T}_{X \to Y}$, $\varphi(\zeta)$ be a non-negative, Schwartz-class function, and the sequence of functions $f_j = s_j \int \varphi(\zeta) e^{i\zeta z} d\zeta$.   
\end{proof}

\begin{lemma}\label{uniformlem}
Let $X$ and $Y$ be Banach spaces, $A$ be a bounded subset of $\R$, $s_0 \in \R$, and $\{T(t,s)\}_{t \in A, s<s_0}$ a set of bounded linear operators $T(t,s): X \rightarrow Y$ satisfying
\begin{equation}
\|T(t,s)\|_{X\rightarrow Y} \le M < \infty \qquad\text{and}\qquad \lim_{s\rightarrow -\infty} \sup_{t\in A} \|T(t,s)f\|_Y = 0
\end{equation}
for any $t\in A$ and $f\in D$ for $D$ a dense subset of $X$. Then, for any $K \subset X$ compact,
\begin{equation}\label{conv0}
\lim_{s\rightarrow -\infty} \sup_{t \in A, f\in K}\|T(t,s)f\|_Y = 0.
\end{equation}
\end{lemma}

\begin{proof}
First, by density of $D\subset X$ and uniform boundedness of the family $\{T(t,s)\}$, certainly
\begin{equation}
\lim_{s\rightarrow -\infty} \|T(t,s)f\|_Y = 0
\end{equation}
for any $f \in X$ and $t\in A$.

Second, fix $K \subset X$ compact and $\epsilon > 0$. Then, $K$ is totally bounded and covered by $N_\epsilon$ balls of radius $\epsilon$, say $\{B_{\epsilon}(f_i)\}_{i=1}^{N_\epsilon}$. 

Now, pick $S(\epsilon) < s_0$ such that for each $1 \le i \le N_\epsilon$, $\|T(t,s)f_i\|_Y < \epsilon$ for $s < S$ and $t \in A$.
Thus, for any $s < S$ and any $f\in K$,
\begin{equation}
\|T(t,s)f\|_Y \le \|T(t,s)f_i\|_Y + \|T(t,s)(f - f_i)\| \le \epsilon + M\epsilon.
\end{equation}
Therefore, taking a supremum over $f\in K$ and $t\in A$, yields (\ref{conv0}).
\end{proof}

\subsection{Linearized Evolution Operators}

Since we wish to work with mild solutions within the self-similar coordinate system, we need to isolate an appropriate part of the equations to treat perturbatively. In this subsection, we describe and estimate various linear evolution operators, which determine the terms we are treating perturbatively. For notational convenience, we first name the various solution operators we will consider.
\begin{itemize}
\item The one parameter semigroup $e^{\tau\L}:B_zL^2(m) \rightarrow B_zL^2(m)$ denotes the solution map to the scalar linear evolution equation on $\R^2 \times \T$,
\begin{equation}
\partial_\tau f = \L f = (\Delta_\xi + \frac{1}{2}\xi \cdot \nabla_\xi + 1)f.
\end{equation}
\item The two parameter semigroup $S_0(\tau,\tau^\prime):B_zL^2(m)\rightarrow B_zL^2(m)$ denotes the solution map to the scalar linearized evolution equation on $\R^2 \times \T$,
\begin{equation}\label{S_0defn}
\partial_\tau f = \L f + e^\tau \partial_{zz}f = (\L + \sdz^2) f.
\end{equation}
\item The one parameter semigroup $\Gamma_\alpha(\tau):B_zL^2(m)^2 \rightarrow B_zL^2(m)^2$ denotes the solution map to the vector-valued linearized evolution equation on $\R^2\times \R$,
\begin{equation}
\partial_\tau f^\xi + \alpha (v^g \cdot \nabla_\xi)f^\xi - \alpha(f^\xi \cdot \nabla_\xi)v^g = \L f^\xi
\end{equation}
\item The one parameter semigroup $\Tau_\alpha(\tau):B_zL^2(m) \rightarrow B_zL^2(m)$ denotes the solution map to the scalar linearized evolution equation on $\R^2\times \R$,
\begin{equation}
\partial_\tau f^z + \alpha v^g \cdot \nabla_\xi f^z + \nabla^\perp_\xi \Delta_\xi^{-1}f^z \cdot \alpha\nabla_\xi G = \L f^z
\end{equation}
\item The two parameter semigroup $S_\alpha(\tau,\tau^\prime):B_zL^2(m)^3 \rightarrow B_zL^2(m)^3$ denotes the solution map to the vector-valued linearized evolution equation on $\R^2 \times \R$,
\begin{equation}
\begin{aligned}\label{Salphadefn}
\partial_\tau f^\xi + \alpha (v^g \cdot \nabla_\xi)f^\xi - \alpha(f^\xi \cdot \nabla_\xi)v^g - \alpha G \sdz v^\xi &= (\L + \sdz^2)f^\xi\\
\partial_\tau f^z + \alpha v^g \cdot\nabla_\xi f^z + \alpha \nabla_\xi G \cdot v^\xi - \alpha G\sdz v^z  &= (\L + \sdz^2)f^z\\
v^\xi =  \sdz (-\slap)^{-1}(f^{\xi})^\perp -& \nabla_\xi^\perp(-\slap)^{-1}f^z\\
v^z = \nabla_\xi^\perp &\cdot (-\slap)^{-1} f^\xi.
\end{aligned}
\end{equation}
\end{itemize}
Now, in Propositions \ref{prop1} and \ref{prop2}, we claim that each of the above solution maps is well defined, and satisfies regularization estimates similar to that of the heat semigroup. In the following, $$a(\tau) = (1 - e^{-\tau})^{-1}\sim \tau^{-1}$$ is the analogue of time in the $\tau \rightarrow 0$ asymptotic regime.

\begin{proposition}[From \cite{gallaywayne1,gallaywayne2,gallaghergallay}]\label{prop1}
Fix $m > 2$. Each of the semigroups $e^{\tau\L}$, $\Gamma_\alpha(\tau)$, and $\Tau_\alpha(\tau)$ are well-defined and strongly continuous on $B_zL^2(m)$ or $B_zL^2(m)^2$. Moreover, for each $\gamma > 0$, $1\le q \le p \le 2$, $f\in B_zL^q(m)$ and $0 < \tau < \infty$, the following regularization estimates hold:
\begin{equation}
\|Q(\tau)f\|_{B_zL^p(m)} \lesssim_{\alpha,m,p,q,\gamma} \frac{e^{\gamma \tau}}{a(\tau)^{\frac{1}{q} - \frac{1}{p}}}\|f\|_{B_zL^q(m)}
\end{equation}
where $Q$ is one of the above semigroups. Furthermore, we have the following gradient estimates:
\begin{equation}
\|\nabla_\xi Q_{sc}(\tau)f\|_{B_zL^p(m)} + e^{\tau/2}\|Q_{sc}(\tau)\nabla_\xi f\|_{B_zL^p(m)} \lesssim_{\alpha,m,p,q,\gamma} \frac{e^{\gamma\tau}}{a(\tau)^{\frac{1}{q} - \frac{1}{p} + \frac{1}{2}}}\|f\|_{B_zL^q(m)}
\end{equation}
where $Q_{sc}$ is one of the scalar semigroups, i.e. $e^{\tau\L}$ or $\Tau_\alpha(\tau)$, and otherwise,
\begin{equation}
\|\nabla_\xi \Gamma_\alpha(\tau)f\|_{B_zL^p(m)} + e^{\tau/2}\|\Gamma_\alpha(\tau)\div_\xi f\|_{B_zL^p(m)} \lesssim_{\alpha,m,p,q,\gamma} \frac{e^{\gamma\tau}}{a(\tau)^{\frac{1}{q} - \frac{1}{p} + \frac{1}{2}}}\|f\|_{B_zL^q(m)}.
\end{equation}
\end{proposition}
\begin{proof}
See \cite{gallaywayne1} for a discussion of $e^{\tau\L}$ on $L^2(m) \subset L^1(\R^2)$ spaces, \cite{gallaywayne2} for the construction of $\Tau_\alpha$ on $L^2(m) \subset L^1(\R^2)$ spaces, and \cite{gallaghergallay} for the gradient estimates on $\Tau_\alpha$. See Appendix B of \cite{jacob} for a construction of $\Gamma_\alpha$ and estimates. The corresponding $B_zL^r(m)$ estimates then follow using Lemma \ref{extensionlem}.
\end{proof}

\begin{proposition}[Essentially from \cite{jacob}]\label{prop2}
Fix $m > 2$. Both of the two-parameter semigroups $S_0(\tau,\tau_0)$ and $S_\alpha(\tau,\tau_0)$ are well-defined and strongly continuous on $B_zL^2(m)$ or $B_zL^2(m)^3$. Moreover, for each $\gamma > 0$, $1 < q \le p \le 2$, and $-\infty < \tau_0 < \tau$, the following regularization estimates hold:
\begin{equation}\label{regularization}
\|Q(\tau,\tau_0)f\|_{B_zL^p(m)} \lesssim_{\alpha,m,p,q,\gamma} \frac{e^{\gamma(\tau - \tau_0)}}{a(\tau - \tau_0)^{\frac{1}{q} - \frac{1}{p}}}\|f\|_{B_zL^q(m)}.
\end{equation}
where $Q$ is $S_0$ or $S_\alpha$.
Furthermore, we have the following gradient estimates for $S_0$: for $f\in B_zL^q(m)$ and $g\in B_zL^q(m)^2$ a vector field,
\begin{equation}\label{scalar1}
\|\sdv_\tau S_0(\tau,\tau_0)f\|_{B_zL^p(m)} \lesssim_{\alpha,m,p,q,\gamma} \frac{e^{\gamma(\tau-\tau_0)}}{a(\tau - \tau_0)^{\frac{1}{q} - \frac{1}{p} + \frac{1}{2}}}\|f\|_{B_zL^q(m)}.
\end{equation}
and
\begin{equation}\label{scalar2}
e^{(\tau-\tau_0)/2}\|S_0(\tau,\tau_0)\div_{\xi}g\|_{B_zL^p(m)} \lesssim_{\alpha,m,p,q,\gamma} \frac{e^{\gamma(\tau-\tau_0)}}{a(\tau - \tau_0)^{\frac{1}{q} - \frac{1}{p} + \frac{1}{2}}}\|g\|_{B_zL^q(m)}
\end{equation}

We also have the following gradient estimates for $S_\alpha$: there exists $\mu = \mu(\alpha)\in (0,1/2)$ for $h\in B_zL^q(m)^3$ a vector field and $F\in B_zL^q(m)^9$ a $3\times 3$ matrix satisfying $\divz\divz F = 0$,
\begin{equation}\label{vector1}
\|\sdv_\tau S_\alpha(\tau,\tau_0)h\|_{B_zL^p(m)} \lesssim_{\alpha,m,p,q} \frac{e^{\mu(\tau -\tau_0)}}{a(\tau - \tau_0)^{\frac{1}{q} - \frac{1}{p} + \frac{1}{2}}}\|h\|_{B_zL^q(m)},
\end{equation}
\begin{equation}\label{vector2}
e^{(\tau - \tau_0)/2}\|S_\alpha(\tau,\tau_0)\divz_{\tau_0}F\|_{B_zL^2(m)} \lesssim_{\alpha,m,q} \frac{e^{\mu(\tau -\tau_0)}}{a(\tau - \tau_0)^{\frac{1}{q}}}\|F\|_{B_zL^q(m)},
\end{equation}
and
\begin{equation}\label{vector3}
e^{(\tau - \tau_0)/2}\|\sdv_\tau S_\alpha(\tau,\tau_0)\divz_{\tau_0}F\|_{B_zL^2(m)}
 \lesssim_{\alpha,m,q} \frac{e^{\mu(\tau -\tau_0)}}{a(\tau - \tau_0)^{\frac{1}{q} + \frac{1}{2}}}\|F\|_{B_zL^q(m)}.
\end{equation}
\end{proposition}

\begin{proof}
The construction of $S_0$ and $S_\alpha$ along with (\ref{regularization}) for $p = q = 2$ are shown in [Proposition 3.2; \cite{jacob}].

The general case of (\ref{regularization}) $S_0$ is due to the forthcoming explicit representation. Say $F(\xi,z)$ is a fixed function and define $F^\prime(\xi,z) = F(e^{-\tau_0/2}\xi,z)$. Then, since $S_0$ is $e^{t\Delta}$ in self-similar coordinates, we obtain
\begin{equation}
S_0(\tau,\tau_0)F_\sigma(\xi,z) = e^{\tau - \tau_0}\left[e^{\left(e^\tau - e^{\tau_0}\right)\Delta}F^\prime\right](\xi e^{\tau/2},z).
\end{equation}
Using the well-known heat kernel in $\R^3$, a change of coordinates, and the explicit representation of $e^{\tau\L}$, found, for example, in [Appendix B.2; \cite{jacob}], we find
\begin{equation}
\begin{aligned}\label{representation}
	S_0(\tau,\tau_0)F(\xi,z) &= \frac{e^{\tau -\tau_0}}{\left(4\pi\left(e^\tau - e^{\tau_0}\right)\right)^{3/2}}\int_{\R^3} \exp\left(-\frac{|e^{\tau/2}\xi - \xi^\prime|^2 + |z - z^\prime|^2}{4(e^\tau - e^{\tau_0})}\right) F\left(\xi^\prime e^{-\tau_0/2},z^\prime\right) \ d\xi^\prime dz^\prime\\
	&=\frac{e^{\tau - \tau_0}\left[4\pi \left(e^\tau - e^{\tau_0}\right)\right]^{-1/2}}{4\pi a(\tau -\tau_0)}\int_{\R^3} \biggr[\exp\left(-\frac{|z - z^\prime|^2}{4(e^\tau - e^{\tau_0})}\right)\\
	&\qquad\qquad\exp\left(-\frac{|\xi - \xi^\prime|^2}{4a(\tau - \tau_0)}\right)F_\sigma\left(\xi^\prime e^{(\tau - \tau_0)/2},z^\prime\right)\biggr]\ d\xi^\prime dz^\prime\\
	&= \frac{1}{\left[4\pi \left(e^\tau - e^{\tau_0}\right)\right]^{1/2}}\int_{\R} \exp\left(-\frac{|z - z^\prime|^2}{4(e^\tau - e^{\tau_0})}\right)e^{(\tau-\tau_0)\L}
F(z^\prime) \ dz^\prime.
\end{aligned}
\end{equation}
Taking the Fourier transform of (\ref{representation}) in $z$ yields
\begin{equation}
\widehat{S_0(\tau,\tau_0)F} = ce^{-(e^\tau - e^{\tau_0})\zeta^2}e^{(\tau-\tau_0)\L}\widehat{F},
\end{equation}
where $c$ is a universal constant. Now, we prove (\ref{regularization}) for $S_0$ using the estimates for $e^{\tau\L}$ from Proposition \ref{prop1}:
\begin{equation}\label{scalarproof}
\begin{aligned}
\|S_0(\tau,\tau_0)F\|_{B_zL^p(m)} &\lesssim \int_{\R} e^{-(e^\tau - e^{\tau_0})\zeta^2}\|e^{(\tau-\tau_0)\L}\widehat{F}\|_{L^p(m)} \ d\zeta \lesssim \int_{\R} \|e^{(\tau-\tau_0)\L}\widehat{F}\|_{L^p(m)}\\
 &\lesssim_{p,q,m,\gamma} \frac{e^{\gamma(\tau - \tau_0)}}{a(\tau - \tau_0)^{\frac{1}{q} - \frac{1}{p}}}\|F\|_{B_zL^q(m)}.
\end{aligned}
\end{equation}

The general case of (\ref{regularization}) for $S_\alpha$ follows from a simple modification of the proof of [Proposition 3.2; \cite{jacob}]. Indeed, as in \cite{jacob}, we take the Fourier transform in $z$, and write the resulting equation as a Duhamel integral of $(\Gamma_\alpha,\Tau_\alpha)$ to obtain
\begin{equation}
\begin{aligned}
w(\tau,\zeta) = e^{-|\zeta|^2(e^\tau - e^{\tau_0})}&\begin{bmatrix}\Gamma_\alpha(\tau-\tau_0) & 0\\0 & \Tau_\alpha(\tau-\tau_0)\end{bmatrix}w(\tau_0,\zeta)\\
	 &+ \alpha\int_{\tau_0}^\tau e^{-|\zeta|^2(\tau - \tau_0)}\begin{bmatrix}\Gamma_\alpha(\tau-s) & 0\\0 & \Tau_\alpha(\tau-s)\end{bmatrix}Z(w(s,\zeta)) \ ds,
\end{aligned}
\end{equation}
where $Z$ contains the perturbative terms
\begin{equation}
Z(w(\tau)) = ie^{\tau/2}\zeta\begin{bmatrix}G(\xi)\left[ie^{\tau/2}\zeta(e^\tau|\zeta|^2 - \Delta_\xi)^{-1}(w^\xi)^\perp -\nabla_\xi^\perp(e^\tau|\zeta|^2 - \Delta_\xi)^{-1}w^z(\tau)\right] \\ \nabla^\perp_\xi \cdot \left(G(\xi)(e^\tau|\zeta|^2 - \Delta_\xi)^{-1}w^\xi\right) + \nabla_\xi G\cdot \left(e^\tau|\zeta|^2 - \Delta_\xi)^{-1} + \Delta_\xi^{-1}\right)\nabla_\xi^\perp \end{bmatrix}.
\end{equation}
In \cite{jacob}, the estimate,
\begin{equation}
\|Z(w(s,\zeta))\|_{L^2(m)} \lesssim_{m,\delta} \left(e^{s/2}|\zeta|\right)^{1 - 2\delta}\|w(s,\zeta)\|_{L^2(m)},
\end{equation}
is shown for each $m > 1$, $0 < \delta < 1/2$ using $L^2(m)$ weighted estimates on the Biot-Savart law in [Lemma 3.3; \cite{jacob}]. In the proof of the weighted estimates, one can replace the embedding $L^2(m) \embeds L^r$ for $1 \le r \le 2$ for $m > 1$ with $L^p(m) \embeds L^r$ for $1\le r \le p$ for $m > 2$ to obtain $L^p(m)$ weighted estimates. These modified estimates are sufficient to show
\begin{equation}
\|Z(w(s,\zeta))\|_{L^p(m)} \lesssim_{m,\delta} \left(e^{s/2}|\zeta|\right)^{1 - 2\delta}\|w(s,\zeta)\|_{L^p(m)}
\end{equation}
for each $0 < \delta < 1 - \frac{1}{p}$ and $1 \le p \le 2$ as in the proof of [Corollary 3.4; \cite{jacob}]. Then, taking $\delta = 1/4$ and using the regularization estimates on $\Gamma_\alpha$ and $\Tau_\alpha$ from Proposition \ref{prop1}, for $\gamma > 0$, we obtain the inequality
\begin{equation}
\begin{aligned}
\|w(\tau,\zeta)\|_{L^p(m)} \lesssim_{m,\gamma,\alpha} e^{-|\zeta|^2(\tau - \tau_0)}&\frac{e^{(\tau - \tau_0)\gamma}}{a(\tau - \tau_0)^{\frac{1}{q} - \frac{1}{p}}}\|w(\tau_0,\zeta)\|_{L^q(m)}\\
	&+ \int_{\tau_0}^\tau e^{-|\zeta|^2(\tau - s)}\left(e^{s/2}|\zeta|\right)^{1/2}e^{\gamma (\tau - s)}\|w(s,\zeta)\|_{L^p(m)}.
\end{aligned}
\end{equation}
Applying Gr\"onwall's inequality to $\tau \mapsto e^{|\zeta|^2\tau}e^{-\gamma\tau}\|w(\tau,\zeta)\|_{L^p(m)}$, yields the estimate
\begin{equation}
\begin{aligned}
\|w(\tau,\zeta)\|_{L^p(m)} &\lesssim_{m,\gamma,\alpha} \frac{e^{(\tau - \tau_0)\gamma}}{a(\tau - \tau_0)^{\frac{1}{q} - \frac{1}{p}}}\|w(\tau_0,\zeta)\|_{L^q(m)} \exp\left(C_\alpha|\zeta|^{1/2}\left(e^{\tau/4} - e^{\tau_0/4}\right) - |\zeta|^2\left(e^{\tau} - e^{\tau_0}\right)\right)\\
	&\lesssim_{m,\gamma,\alpha} \frac{e^{(\tau - \tau_0)\gamma}}{a(\tau - \tau_0)^{\frac{1}{q} - \frac{1}{p}}}\|w(\tau_0,\zeta)\|_{L^q(m)}
\end{aligned}
\end{equation}
which directly implies (\ref{regularization}).

Second, the gradient estimates (\ref{scalar1}) and (\ref{scalar2}) are shown in [Section 3.3; \cite{jacob}] for general $p$ and $q$ for short time. The estimates  (\ref{scalar1}) and (\ref{scalar2}) for general $p$ and $q$ and long time follow from (\ref{representation}) and the corresponding estimate on $e^{\tau\L}$ from Proposition \ref{prop1} in the same manner as in (\ref{scalarproof}).

Third, the gradient estimate (\ref{vector1}) is shown for short times in [Lemma 3.6; \cite{jacob}] for $p = q = 2$ by writing $S_\alpha$ as a Duhamel integral of $S_0$, and using the estimates on $S_0$ and a contraction mapping argument on a suitable Banach space $X$ encoding the desired estimate. One can modify the choice of $X$, so that $X_{p,q}$ is the closed subspace of $C([\tau_0,\tau_0 + \delta];B_zL^q(m))$ such that the norm,
\begin{equation}
\begin{aligned}
\|\Omega\|_{X_{p,q}} = \sup_{\tau_0 < \tau < \tau_0 + \delta} &\bigg(\|\Omega(\tau)\|_{B_zL^q(m)} + a(\tau - \tau_0)^{\frac{1}{q} - \frac{1}{p}}\|\Omega\|_{B_zL^p(m)}\\
	&+ a(\tau-\tau_0)^{\frac{1}{q} - \frac{1}{p} + \frac{1}{2}}\|\sdv_\tau \Omega(\tau)\|_{B_zL^p(m)}\bigg),
\end{aligned}
\end{equation}
is finite. This modification still yields the estimate
\begin{equation}
\|S_\alpha(\tau,\tau_0)F\|_{X_{p,q}} \lesssim_{p,q,m,\alpha} \|F\|_{B_zL^q(m)} + \delta^{1- \frac{1}{p}}\|S_\alpha(\tau,\tau_0)F\|_{X_{p,q}}.
\end{equation}
Taking $\delta$ sufficiently small, depending only on $\alpha$, $m$, $p$, and $q$, yields the short time estimates for general $p$ and $q$. Then, [Section 3.4; \cite{jacob}] explains how to combine the long time estimates (\ref{regularization}) and the short time smoothing estimates (\ref{vector1}) to obtain (\ref{vector1}) for all time.

Finally, (\ref{vector2}) and (\ref{vector3}) are shown in [Theorem 3.1; \cite{jacob}] for $p = 2$ and $1 \le q \le 2$.
\end{proof}

\subsection{The Vorticity Equation on $\R^2$}

We will need properties of the two dimensional self-similar vorticity equation, namely properties of
\begin{equation}\label{svorticity2}
\partial_t \omega + \nabla_\xi^\perp \Delta_\xi^{-1} \omega \cdot \nabla_\xi \omega = \L\omega, 
\end{equation}
where $\omega: \R^2 \times \R^+ \rightarrow \R$ is the scalar vorticity. It is known that the (\ref{svorticity2}) is well-posed on $L^2(m)$ and that the invariant sets of the flow on $L^2(m)$ are quite rigid.

\begin{theorem}[(From \cite{gallaywayne1,gallaywayne2})]\label{thm4}
For any $m > 1$, (\ref{svorticity2}) generates a strongly continuous flow $\phi$ on $L^2(m)$. Moreover, if $\A\subset L^2(m)$ is fully invariant under $\phi$ in the sense that $\phi(\tau)\A = \A$ for any $\tau > 0$, then
\begin{equation}
\A \subset \{\beta G \ | \ \beta\in \R\}.
\end{equation}
\end{theorem}
For a proof and discussion of the well-posedness of (\ref{svorticity2}) see \cite{gallaywayne1}. For a proof of the rigidity of invariant sets see \cite{gallaywayne2}.

Now, since we will be working on $B_zL^2(m)$ we want (\ref{svorticity2}) to generate a flow $\Phi$ on $B_zL^2(m)$. That is, for any initial data $w_0\in B_zL^2(m)$, we want a unique global in time solution $w\in C([0,\infty);B_zL^2(m))$ to
\begin{equation}\label{mild2dform1}
w(\tau) = e^{\tau\L}w_0 - \int_0^\tau e^{(\tau - s)\L}(\nabla^\perp_\xi \Delta_\xi^{-1}w(s) \cdot \nabla_\xi w(s)) \ ds.
\end{equation}
While this seems reasonable to expect, we are only able to show local in time well-posedness, which is sufficient for our purposes.

\begin{lemma}\label{flowlem}
Fix $m > 2$. For any $K \subset B_zL^2(m)$ compact, there exists a $T = T(K) > 0$ such that for any $w_0\in K$, there is a unique mild solution $w \in C([0,T];B_zL^2(m))$ to (\ref{mild2dform1}) with initial data $w_0$.
Moreover, $w(\tau)$ is given by
\begin{equation}\label{commutation}
w(\tau,\xi,z) = (\phi(\tau)w_0(\cdot,z))(\xi),
\end{equation}
where $\phi$ is the flow generated by (\ref{svorticity2}) on $L^2(m)$.
\end{lemma}

\begin{proof}
This follows from a contraction mapping argument in the space $C([0,T_{w_0}];B_zL^2(m))$. Set $N(K) = \max_{w_0\in K} \|w_0\|_{B_zL^2(m)}$. Fix $w_0\in K$ and let $\Psi$ be the operator, 
\begin{equation}
\Psi(w)(\tau) = e^{\tau\L}w_0 + \int_0^\tau e^{(\tau-s)\L}(\nabla_\xi^\perp\Delta_\xi^{-1}w(s) \cdot \nabla_\xi w(s)) \ ds,
\end{equation}
and define
$$B_{T,\epsilon} = \set{ w \in C([0,T];B_zL^2(m)) : \sup_{0 < \tau < T}\norm{w - w_0}_{B_z L^2(m)} \le 2\epsilon }. $$
Then, we estimate the nonlinearity using the linear propagator estimates in Proposition \ref{prop1}, H\"older's inequality and the boundedness of the 2d Biot-Savart law (from 2d estimates and Lemma \ref{extensionlem}), and the embedding $B_zL^2(m) \embeds B_zL^{4/3}$:
\begin{equation}
\begin{aligned}
\biggr\|\int_0^\tau e^{(\tau-s)\L}\nabla_\xi &\cdot \left(w(s)\nabla_\xi^\perp\Delta_\xi^{-1}w(s)\right) \ ds\biggr\|_{B_zL^2(m)}\\
&\hspace{1cm}\lesssim_m \int_0^\tau \frac{e^{\gamma(\tau - s)}}{a(\tau-s)^{3/4}}\|w(s)\nabla_\xi^\perp\Delta_\xi^{-1}w(s)\|_{B_zL^{4/3}(m)} \ ds\\
	&\hspace{1cm} \lesssim_m \int_0^\tau \frac{e^{\gamma (\tau - s)}}{a(\tau-s)^{3/4}}\|w(s)\|_{B_zL^2(m)}\|\nabla_\xi^\perp\Delta_\xi^{-1}w(s)\|_{B_zL^4} \ ds\\
	&\hspace{1cm} \lesssim_m \int_0^\tau \frac{e^{\gamma (\tau - s)}}{a(\tau-s)^{3/4}}\|w(s)\|_{B_zL^2(m)}^2 \ ds.
\end{aligned}
\end{equation}
Then, since $K$ is compact and $e^{\tau\L}$ is a strongly continuous semigroup on $B_zL^2(m)$, there exists a time $T(K)$ such that
$$\sup_{w_0 \in K, \tau < T(K)} \|e^{\tau\L}w_0 - w_0\|_{B_zL^2(m)} < \epsilon.$$
Also, take $T(K)$ and $\epsilon$ sufficiently small so that
$$C_m(\epsilon + N)^2\int_0^T \frac{e^{\gamma (T - s)}}{a(T-s)^{3/4}} \ ds < \epsilon.$$
For such $T,\epsilon$, $\Psi(B_{T,\epsilon}) \subset B_{T,\epsilon}$. The contraction property on $B_{T,\epsilon}$ is proved similarly. Therefore, there is a unique mild solution to (\ref{mild2dform1}) in $B_{T,\epsilon}$, where $T$, $\epsilon$ are chosen uniformly for $w_0\in K$.
Furthermore, taking $v\in C([0,T];B_zL^2(m))$ a mild solution to (\ref{mild2dform1}) and letting $t\rightarrow 0^+$, it follows that $v \in B_{\delta,\epsilon}$ for $\delta$ sufficiently small, which implies $v = w$. Finally, (\ref{commutation}) follows since all of the operators in (\ref{mild2dform1}) commute with evaluation in $z$.
\end{proof}


\section{Proof of Theorem 1}

\subsection{Proof Sketch}

The main idea of the proof of Theorem \ref{thm1} is to adapt a two dimensional compactness-rigidity argument from the proof of \big[Proposition 4.5; \cite{gallaghergallay}\big].
In self-similar coordinates, $w^z$ satisfies (in the mild sense) the scalar equation
\begin{equation}\label{heuristiceqn}
\partial_\tau w^z + \left(\nabla^\perp_\xi (\slap)^{-1} w^z\right) \cdot \nabla_\xi w^z + R(\tau) = \L w + \sdz^2 w^z,
\end{equation}
where $\L$ is the Fokker-Planck operator defined in (\ref{FokkerPlanck}) and our remainder terms are given as
\begin{equation}\label{Rdefn1}
R(\tau) = \left(v^\xi - \left(\nabla^\perp_\xi (\slap)^{-1}w^z\right)\right) \cdot \nabla_\xi w^z + v^z \sdz w^z - w\cdot \sdv v^z,
\end{equation}
the additional terms in the 3d vorticity equation.
We proceed in four steps analogous to the proof in \cite{gallaghergallay}:
\begin{itemize}
\item First, we will use parabolic regularity combined with our assumptions to show that the trajectory $\{w^z(\tau)\}_{\tau < \tau^*}$ is precompact in $B_zL^2(m)$. Hence, if $\A\subset B_zL^2(m)$ is the $\alpha$-limit set of the trajectory $\{w^z(\tau)\}_{\tau < \tau^*}$, then $\A$ is compact and nonempty.
\item Second, we will show that the remainder terms vanish in some norm, $\|R(\tau)\|_X \rightarrow 0$ as $\tau \rightarrow -\infty$.
Combined with the asymptotic convergence of operators, $\slap \rightarrow \Delta_\xi$ and linear propagators $S_0 \rightarrow e^{\tau\L}$, we will conclude that $w^z$ asymptotically (in the $\tau \rightarrow -\infty$ regime) satisfies the 2d self-similar vorticity equation.
\item Third, we will use steps 1 and 2 to show that the $\alpha$-limit set $\A$ is invariant under the flow $\Phi(\tau):B_zL^2(m) \rightarrow B_zL^2(m)$ generated by the 2d vorticity equation.
\item Fourth, we will use the rigidity of invariant sets of $\phi$ shown in \cite{gallaywayne2} to show that $\A = \{\alpha G\}$. This implies $w^z(\tau) \rightarrow \alpha G$ in $B_zL^2(m)$ as $\tau \rightarrow -\infty$ and $w$ lies within the uniqueness class shown in \cite{jacob}.
\end{itemize}

\subsection{Step 1: Parabolic Regularity and Compactness} \label{sec:parabolic}

In this section, we study the regularity of mild solutions to the self-similar vorticity equation (\ref{svorticity}). That is, we study functions,
$$w \in C((-\infty,\tau^*);B_zL^2(m)) \cap L^\infty((-\infty,\tau^*);B_zL^2(m))$$
satisfying the integral equation,
\begin{equation}\label{mildsvorticity}
\begin{aligned}
w(\tau) &= S_0(\tau,\tau_0)w(\tau_0) - \int_{\tau_0}^\tau S_0(\tau,s)\left[(v(s) \cdot \sdv_s) w(s) - (w(s) \cdot \sdv_s) v(s)\right] \ ds\\
	v(\tau) &= -\sdv \times (\slap)^{-1} w(\tau),
\end{aligned}
\end{equation}
for each $\tau_0 < \tau < \tau^*$, where we recall $S_0$ is the linear propagator defined in (\ref{S_0defn}). 
The main result of this section is the following proposition for arbitrary (in the sense that no reference is made to initial data) mild solutions in the sense of (\ref{mildsvorticity}).
\begin{proposition}\label{prop5}
Let $w(\tau)$ be a solution to (\ref{mildsvorticity}) on $\R^2 \times \T$ such that for some $-\infty < \tau^* \le \infty$, $m > 2$, and $\Lambda > 0$,
\begin{enumerate}[label=(\roman*)]
\item \label{assumptioni} $\sup_{\tau < \tau^*}\norm{w(\tau)}_{B_z L^2(m)} \le \Lambda;$
\item \label{assumptionii} $\lim_{\tau \to -\infty}\norm{w^\xi(\tau)}_{B_z L^2(m)} = 0$;
\item \label{assumptioniii}  and $\sup_{\tau < \tau^*}\norm{\partial_z w^z(\tau)}_{B_z L^2(m)} < \infty$.
\end{enumerate}
Then,
\begin{align}
&\sup_{\tau < \tau^*} \|\sdv_\tau w(\tau)\|_{B_zL^2(m)} \le C(\Lambda, m),\label{eqn1}\\
&\lim_{\tau \rightarrow -\infty} \|\sdv_\tau w^\xi(\tau)\|_{B_zL^2(m)} = 0\label{eqn2},
\end{align}
and moreover, the trajectory $\{w^z(\tau)\}_{\tau < \tau^*}$ is precompact in $B_zL^2(m^\prime)$ for each $2 < m^\prime < m$.
\end{proposition}

\begin{remark}
The boundedness of $\sdv w(\tau)$ in (\ref{eqn1}) relies only on Assumption \ref{assumptioni} while the decay of $\sdv w^\xi$ in (\ref{eqn2}) depends on both Assumption \ref{assumptioni} and Assumption \ref{assumptionii}. See Lemma \ref{lem7} and Lemma \ref{lem6} below.
\end{remark}

We prove this proposition by examining the implications of various combinations of the Assumptions \ref{assumptioni}, \ref{assumptionii}, and \ref{assumptioniii} in the sequence of lemmas below.

\begin{lemma}\label{lem7}
Suppose $w$ is a mild solution to (\ref{mildsvorticity}) on $\R^2 \times \T$ such that for some $m > 2$, $-\infty < \tau^* \le  \infty$, and $\Lambda > 0$, 
\begin{equation}
\sup_{\tau < \tau^*} \|w(\tau)\|_{B_zL^2(m)} \le \Lambda < \infty.
\end{equation}
Then, there is a constant $C = C(m,\Lambda)$, such that
\begin{equation}
\sup_{\tau < \tau^*} \|\sdv_\tau w(\tau)\|_{B_zL^2(m)} \le C.
\end{equation}
\end{lemma}

\begin{proof}
We have that for each $\tau_0 < \tau$, $\sdv w$ satisfies the integral equation
\begin{equation}
\sdv_\tau w(\tau) = \sdv_\tau S_0(\tau,\tau_0)w(\tau_0) - \int_{\tau_0}^\tau \sdv_\tau S_0(\tau,s)\left[(v \cdot \sdv_s)w - (w\cdot \sdv_s)v\right](s) \ ds.
\end{equation}
Fix $\gamma > 0$ and $T > 0$ so that $\tau < \tau_0 + T$. The gradient bounds in Proposition \ref{prop2} imply
\begin{equation}
\|\sdv_\tau w(\tau)\|_{B_zL^2(m)} \lesssim_{m} \frac{\Lambda e^{T\gamma}}{a(\tau-\tau_0)^{1/2}} + \int_{\tau_0}^\tau \frac{e^{T\gamma}}{a(\tau - s)^{3/4}}\|(v \cdot \sdv_s)w - (w\cdot \sdv_s)v\|_{B_zL^{4/3}(m)} \ ds.
\end{equation}
Now, we bound the transport terms via H\"older (\ref{holder}), the embedding $B_zL^2(m) \embeds B_zL^1$, and the boundedness of the Biot-Savart law (\ref{biotsavart}).
\begin{equation}
\begin{aligned}
\|(v\cdot \sdv) w\|_{B_zL^{4/3}(m)} &\lesssim_{m} \|v\|_{B_zL^4}\|\sdv w\|_{B_zL^2(m)}\\
	&\lesssim_{m} \|w\|_{B_zL^{4/3}}\|\sdv w\|_{B_zL^2(m)} \lesssim_m \Lambda\|\sdv w\|_{B_zL^2(m)}.
\end{aligned}
\end{equation}
Similarly, we bound the vortex stretching terms via H\"older (\ref{holder}), the boundedness of the Riesz transforms (\ref{riesz}), and the Sobolev embedding $B_z\dot{W}^{1,4/3} \embeds B_zL^4$,
\begin{equation}
\begin{aligned}
\|(w\cdot \sdv) v\|_{B_zL^{4/3}(m)} &\lesssim_m \|w\|_{B_zL^2(m)}\|\sdv v\|_{B_zL^4}\\
	&\lesssim_{m} \|w\|_{B_zL^2(m)}\|w\|_{B_zL^4}\\
	&\lesssim_m \|w\|_{B_zL^2(m)}\|\nabla_\xi w\|_{B_zL^{4/3}}\\
	&\lesssim_{m} \Lambda\|\nabla_\xi w\|_{B_zL^2(m)}.
\end{aligned}
\end{equation}
Thus, we obtain for each $\tau_0 < \tau < \tau_0 + T$, and some $C > 0$ depending only on $\Lambda$ and $m$,
\begin{equation}
\|\sdv_\tau w(\tau)\|_{B_zL^2(m)} \le \frac{Ce^{T\gamma}}{a(\tau-\tau_0)^{1/2}} + Ce^{T\gamma}\int_{\tau_0}^\tau \frac{\|\sdv_s w(s)\|_{B_zL^2(m)}}{a(\tau-s)^{3/4}} \ ds.
\end{equation}
Since $a(\tau - \tau_0)$ is unbounded for $\tau$ close to $\tau_0$, we write 
\begin{multline}
\sup_{\tau_0 < \tau < \tau_0 + T}a(\tau-\tau_0)^{1/2}\|\sdv_\tau w(\tau)\|_{B_zL^2(m)} \le\\ Ce^{T\gamma} + Ce^{T\gamma}I(\tau_0,T)\sup_{\tau_0 < \tau < \tau_0 + T} \left(a(\tau-\tau_0)^{1/2}\|\sdv_\tau w(\tau)\|_{B_zL^2(m)}\right),\label{eqn3}
\end{multline}
where
\begin{equation}
I(\tau_0,T) = \sup_{\tau_0 < \tau < \tau_0 + T} \int_{\tau_0}^\tau \frac{a(\tau-\tau_0)^{1/2}}{a(\tau-s)^{3/4}a(s-\tau_0)^{1/2}} \ ds.
\end{equation}
After a change of variables, we obtain
\begin{equation}
I(\tau_0,T) = \sup_{\tau_0 < \tau < \tau_0 + T} \int_{0}^{\tau -\tau_0} \frac{a(\tau-\tau_0)^{1/2}}{a(\tau - \tau_0 -s)^{3/4}a(s)^{1/2}}\ ds = \sup_{0 < \tau < T} \int_0^\tau \frac{a(\tau)^{1/2}}{a(\tau -s)^{3/4}a(s)^{1/2}} \ ds.\label{I1}
\end{equation}
Thus, recalling that $a(\tau)$ is of order $\tau$ for $\tau \ll 1$, by a change of variables,
\begin{equation}
\int_0^\tau \frac{\tau^{1/2}}{(\tau -s)^{3/4}s^{1/2}} \ ds = \int_0^1 \frac{\tau}{s^{1/2}\tau^{3/4}(1 - s)^{3/4}} \ ds \approx \tau^{1/4}.\label{I2}
\end{equation}
Together (\ref{I1}) and (\ref{I2}) then imply that $I(\tau_0,T)\approx T^{1/4}$ uniformly in $\tau$ and $\tau_0$ in the $T \ll 1$ asymptotic regime. In particular, for $T$ sufficiently small depending only on $m$ and $\Lambda$,
\begin{equation}
\sup_{\tau_0 < \tau < \tau_0 + T} a(\tau - \tau_0)^{1/2}\|\sdv_\tau w(\tau)\|_{B_zL^2(m)} \le 2Ce^{T\gamma},
\end{equation}
independent of $\tau_0$. Thus, we obtain the uniform bound, 
$$\|\sdv w(\tau)\|_{B_zL^2(m)} \le  \frac{2Ce^{T\gamma}}{a(T)^{1/2}},$$ 
for $T$ sufficiently small.
\end{proof}

Now, combining Assumption \ref{assumptioni} with Assumption \ref{assumptioniii}, we obtain the following compactness result. Note, we are still discussing mild solutions to (\ref{mildsvorticity}) without reference to the initial datum.
\begin{lemma}\label{lem8}
Suppose $w$ is a mild solution to (\ref{mildsvorticity}) on $\R^2 \times \T$ such that for some $m > 2$ and $-\infty < \tau^* \le  \infty$, 
\begin{equation}
\sup_{\tau < \tau^*} \|w(\tau)\|_{B_zL^2(m)} + \|\partial_z w^z(\tau)\|_{B_zL^2(m)} < \infty.
\end{equation}
Then, the trajectory $\{w^z(\tau)\}_{\tau < \tau^*}$ is precompact in $B_zL^2(m^\prime)$ for each $2 < m^\prime < m$.
\end{lemma}

\begin{remark}
In the corresponding step in the 2d case \big[Lemma 6.1; \cite{gallaghergallay}\big], the gain in regularity from preceding lemma and the compact embedding $H^1(m+\epsilon) \embeds L^2(m)$ are sufficient to conclude compactness of the trajectory. However, in our case, we do not gain the same amount of regularity in the $z$-direction as in the $\xi$-directions. In particular, Lemma \ref{lem7} implies the bound
\begin{equation}\label{bound}
\|\partial_z w(\tau)\|_{B_zL^2(m)} \le e^{-\tau/2}C(m,\|w\|_{L^\infty_\tau B_zL^2(m)}),
\end{equation}
which degenerates as $\tau\rightarrow -\infty$. The role of Assumption \ref{assumptioniii} is to replace (\ref{bound}) with a bound that does not degenerate at $\tau \rightarrow -\infty$ so that we may gain compactness of the trajectory. This is the only place in our proof where we need to use this assumption.
\end{remark}

\begin{proof}[Proof of Lemma \ref{lem8}]
Fix $2 < m^\prime < m$. Take $\{\tau_n\}_{n=1}^\infty$ an arbitrary sequences of times and let $\{f_n\} \subset L^1_\zeta L^2_\xi(m)$ be the sequence of functions $f_n(\xi,\zeta) = \hat w^z(\tau_n,\xi,\zeta)$ where $\zeta \in \Z$ is the frequency variable corresponding to the Fourier transform in the $z$ variable. Then, by the definition of the $B_zX$ spaces and Lemma \ref{lem7}, $\{f_n(\zeta)\}_{\zeta,n}$ is uniformly bounded in $H^1_\xi(m)$.  Moreover, the boundedness of $\partial_z w^z(\tau_n)$ implies
$$\|\partial_z w^z(\tau_n)\|_{B_zL^2(m)} = \sum_{\zeta\in \Z} |\zeta| \|f_n(\zeta)\|_{L^2(m)} \le C(m) < \infty,$$
and the sequence of functions $\{\zeta f_n(\zeta)\}_{n=1}^\infty$ is uniformly bounded in $L^1_\zeta L^2_\xi(m)$. 

First, we use that the embedding $H^1(m) \embeds L^2(m^\prime)$ is compact (Lemma \ref{lem3}) and a diagonalization argument to produce a subsequence that we call $\{f_n(\zeta,\xi)\}_{n=1}^\infty$ such that for each $\zeta \in \Z$,
$$\lim_{n \rightarrow \infty} \|f_n(\zeta,\cdot) - f(\zeta,\cdot)\|_{L^2(m)} = 0,$$
for some function $f(\zeta,\xi) \in L^1_\zeta L^2_\xi(m)$.

Second, we show that $f_n \rightarrow f$ in $L^1_\zeta L_\xi^2(m^\prime)$.
To that end, fix $\epsilon > 0$ and let $\zeta_0 > 0$ arbitrary. Then,
\begin{equation}
\sum_\zeta \|f_n(\zeta) - f(\zeta)\|_{L^2(m^\prime)} \le \sum_{|\zeta| > \zeta_0} \|f(\zeta)\|_{L^2(m^\prime)} + \sum_{|\zeta| > \zeta_0} \|f_n(\zeta)\|_{L^2(m^\prime)} + \sum_{|\zeta| \le \zeta_0} \|f_n - f\|_{L^2(m^\prime)}.
\end{equation}
By Fatou's lemma, $f$ is summable and the first term may be made arbitrarily small by taking $\zeta_0$ large. By the uniform boundedness of $\zeta f_n(\zeta)$, we have tightness and the the second term may be made arbitrarily small uniformly in $n$ by taking $\zeta_0$ large. Finally, using the finiteness of the sum and pointwise convergence, the third term may be made arbitrarily small by taking $n$ large. Thus,
\begin{equation}
\lim_{n\rightarrow \infty} \|f_n - f\|_{L^1_\zeta L_\xi^2(m^\prime)} \le \epsilon
\end{equation}
for all $\epsilon > 0$ and $f_n$ converges to $f$ in $L_\zeta^1L_\xi^2(m^\prime)$. Letting $w^\infty = \check f$, the inverse Fourier transform (in $\zeta$) of $f$, we have shown that $w^z(\tau_n) \rightarrow w^\infty$ in $B_zL^2(m^\prime)$. Since $\{\tau_n\}_{n=1}^\infty$ was arbitrary, the trajectory $\{w^z(\tau)\}_{\tau < \tau^*}$ is precompact in $B_zL^2(m^\prime)$.
\end{proof}

We now want to show that Assumption \ref{assumptionii} implies the situation is better for $w^\xi$ than $w^z$. Indeed, assuming $w^\xi(\tau)$ vanishes backwards in time, we will use parabolic regularity to say the derivatives $\sdv w^\xi(\tau)$ vanish as well.
We begin by bounding the nonlinearity in the $w^\xi$ equations.
\begin{lemma}\label{lem12}
Suppose $w$ is a mild solution to (\ref{mildsvorticity}) on $\R^2\times \T$ such that for some $m > 2$, $-\infty < \tau^* \le \infty$, and $\Lambda > 0$,
\begin{equation}\label{energybound}
\sup_{\tau < \tau^*} \|w(\tau)\|_{B_zL^2(m)} \le \Lambda.
\end{equation}
Then,
\begin{equation}\label{bilinear_est}
\|(v\cdot \sdv) w^\xi - (w\cdot \sdv)v^\xi\|_{B_zL^{4/3}(m)} \lesssim_{m,\Lambda} \|w^\xi\|_{B_zL^2(m)} + \|\sdv w^\xi\|_{B_zL^2(m)},
\end{equation}
where the implicit constant depends on $\Lambda$ and $m$.
\end{lemma}

\begin{remark}
Note that there is no $w^z$ in the right hand side of the estimate (\ref{bilinear_est}). This is because the only appearance of $w^z$ in the $w^\xi$ equations is contained in the nonlinear terms estimated in this lemma. We can bound any term containing $w^z$ in terms of $w^\xi$ either by using (\ref{energybound}) or by relating to $w^\xi$ by the divergence-free condition.
\end{remark}

\begin{proof}[Proof of Lemma \ref{lem12}]
To estimate the transport term, we use H\"older's inequality, the boundedness of the Biot-Savart law, and the embedding $L^{4/3}\embeds L^2(m)$ for $m > 1$ to compute
\begin{equation}
\begin{aligned}
\|(v\cdot \sdv)w^\xi\|_{B_zL^{4/3}(m)} &\lesssim_m \|v\|_{B_zL^4}\|\sdv w^\xi\|_{B_zL^2(m)}\\
	&\lesssim_m \|w\|_{B_zL^{4/3}}\|\sdv w^\xi\|_{B_zL^2(m)}\\
	&\lesssim_m \|w\|_{B_zL^2(m)}\|\sdv w^\xi\|_{B_zL^2(m)} \lesssim_{m} \Lambda\|\sdv w^\xi\|_{B_zL^2(m)}.
\end{aligned}
\end{equation}

Estimating the vortex stretching term is slightly more delicate and uses more structure of the equations. In particular, we split the velocity using the structure of the Biot-Savart law (\ref{biotsavartstructure})
\begin{equation}\label{splitting1}
(w \cdot \sdv) v^\xi = (w\cdot \sdv)(\sdz(\slap)^{-1}(w^\xi)^\perp) + (w\cdot \sdv)(\nabla_\xi^\perp (\slap)^{-1}w^z). 
\end{equation}
The first term in (\ref{splitting1}) is bounded via H\"older's inequality, the boundedness of Riesz transforms, and the Sobolev embedding $B_z\dot{W}^{1,4/3} \embeds B_zL^4$, and the embedding $B_zL^2(m) \embeds B_zL^{4/3}$ as
\begin{equation}
\|(w\cdot \sdv)(\sdz(\slap)^{-1}(w^\xi)^\perp)\|_{B_zL^{4/3}(m)} \lesssim_m \Lambda\|w^\xi\|_{B_zL^4} \lesssim_m \Lambda \|\nabla_\xi w^\xi\|_{B_zL^2(m)}.
\end{equation}
However, since $w^z$ does not appear on the right hand side of the desired inequality, we must analyze the second term in (\ref{splitting1}) more carefully. 
We decompose the second term in (\ref{splitting1}) further as
\begin{equation}\label{splitting2}
(w\cdot \sdv)(\nabla_\xi^\perp (\slap)^{-1}w^z) = (w^\xi \cdot \nabla_\xi)(\nabla_\xi^\perp (\slap)^{-1}w^z) + w^z\sdz(\nabla_\xi^\perp (\slap)^{-1}w^z).
\end{equation}
Again, the first term in (\ref{splitting2}) is bounded via H\"older's inequality, boundedness of Riesz transforms, the Sobolev embedding $B_zH^1\embeds B_zL^4$, and Lemma \ref{lem7} as
\begin{equation}
\|(w^\xi \cdot \nabla_\xi)(\nabla_\xi^\perp (\slap)^{-1}w^z)\|_{B_zL^{4/3}(m)} \lesssim_m \|w^\xi\|_{B_zL^2(m)}\|w^z\|_{B_zL^4} \lesssim_{m,\Lambda} \|w^\xi\|_{B_zL^2(m)}.
\end{equation}
For the second term in (\ref{splitting2}), we commute derivatives and rewrite it as
\begin{equation}
w^z\sdz\nabla_\xi^\perp (\slap)^{-1}w^z = -w^z\nabla_\xi^\perp \nabla_\xi \cdot (\slap)^{-1} w^\xi
\end{equation}
where we have used the divergence-free condition
\begin{equation}
\sdz w^z = -\nabla_\xi \cdot w^\xi.
\end{equation}
The resulting expression is bounded again using H\"older's inequality, boundedness of Riesz transforms, and the Sobolev embedding $B_z\dot{H}^1(m)\embeds B_zL^4$ as
\begin{equation}
\|w^z\nabla_\xi^\perp \nabla_\xi \cdot \slap^{-1} w^\xi\|_{B_zL^{4/3}(m)} \lesssim_m \Lambda\|\nabla_\xi w^\xi\|_{B_zL^2(m)}.
\end{equation}
This completes the proof as we have now bounded all terms in the vortex stretching term via the successive splittings (\ref{splitting1}) and (\ref{splitting2}).
\end{proof}

\begin{lemma}\label{lem6}
Suppose $w$ is a mild solution to (\ref{mildsvorticity}) on $\R^2\times \T$ such that for some $m > 2$, $\tau^* > -\infty$, and $\Lambda > 0$,
\begin{equation}
\sup_{\tau < \tau^*} \|w(\tau)\|_{B_zL^2(m)} \le \Lambda.
\end{equation}
Suppose also that $w^\xi\rightarrow 0$ strongly in $B_zL^2(m)$. Then, $\sdv w^\xi \rightarrow 0$ and $\sdz w^z \rightarrow 0$ in $B_zL^2(m)$.
\end{lemma}

\begin{remark}
Note, that we are not assuming Assumption \ref{assumptioniii} in this lemma, but we still obtain the decay of $\sdz w^z = e^{\tau/2} \partial_z w^z$. The decay of $\sdz w^z$ is sufficient to replace $\partial_z w^z$ bounded for all later steps in the proof of Theorem \ref{thm1}. This emphasizes that the role of Assumption \ref{assumptioniii} is only to gain compactness of the trajectory $\{w^z(\tau)\}$. 
\end{remark}

\begin{proof}[Proof of Lemma \ref{lem6}]
The idea of the proof is to use parabolic regularity as in Lemma \ref{lem7}. We recall $\sdv w^\xi$ satisfies the integral equation
\begin{equation}
\sdv_\tau w^\xi(\tau) = \sdv_\tau S_0(\tau,\tau_0)w^\xi(\tau_0) - \int_{\tau_0}^\tau \sdv_\tau S_0(\tau,s)\left[(v\cdot \sdv_s)w^\xi - (w \cdot \sdv_s)v^\xi\right](s) \ ds,
\end{equation}
for each $\tau_0 \le \tau$. Fixing $\gamma > 0$ and letting $\tau < \tau_0  + T$, we use the gradient estimates on $S_0$ and Lemma \ref{lem12} to obtain the estimate
\begin{equation}
\begin{aligned}
\|\sdv_\tau w^\xi(\tau)\|_{B_zL^2(m)} \le &\frac{Ce^{T\gamma}}{a(\tau-\tau_0)^{1/2}}\|w^\xi(\tau_0)\|_{B_zL^2(m)} + Ce^{T\gamma}\int_{\tau_0}^\tau \frac{\|w^\xi(s)\|_{B_zL^2(m)}}{a(\tau - s)^{3/4}} \ ds\\
 &\quad+ Ce^{T\gamma}\int_{\tau_0}^\tau \frac{\|\sdv_s w^\xi(s)\|_{B_zL^2(m)}}{a(\tau-s)^{3/4}} \ ds,
\end{aligned}
\end{equation}
where $C$ depends only on $m$ and $\Lambda$.
A slight modification of the argument in Lemma \ref{lem7} to account for the additional inhomogeneous terms, yields the bound
\begin{align}\label{ineq1}
\|\sdv w^\xi(\tau_0 + T)\|_{B_zL^2(m)}\le \frac{2Ce^{T\gamma}\left[1 + C^\prime a(T)^{1/2}T^{1/4}\right]}{a(T)^{1/2}}\sup_{\tau_0 < s < \tau_0 + T} \|w^\xi(\tau)\|_{B_zL^2(m)}
\end{align}
where $C^\prime > 0$ is a universal constant and $T > 0$ is small enough such that
$$Ce^{\gamma T}\int_0^T \frac{a(\tau)^{1/2}}{a(s)^{1/2}a(\tau - s)^{3/4}}\ ds < \frac{1}{2}.$$
Inequality (\ref{ineq1}) and $w^\xi(\tau) \rightarrow 0$ in $B_zL^2(m)$ together directly imply $\sdv w^\xi(\tau) \rightarrow 0$ in $B_zL^2(m)$ as desired. We now note that $w^z$ is related to $w^\xi$ via the divergence-free condition,
$$\sdv \cdot w = \nabla_\xi \cdot w^\xi + \sdz w^z = 0.$$
As a consequence, $\nabla_\xi w^\xi \rightarrow 0$ implies that $\sdz w^z \rightarrow 0$ in $B_zL^2(m)$ as claimed.
\end{proof}

\subsection{Step 2: Error Estimates}
First, we estimate the 3d error terms collected in $R(\tau)$ in (\ref{Rdefn1}), which we recall are
\begin{equation}\label{R}
R(\tau) = \left(v^\xi - \left(\nabla^\perp_\xi (\slap)^{-1}w^z\right)\right) \cdot \nabla_\xi w^z + v^z \sdz w^z - w\cdot \sdv v^z.
\end{equation}
Second, we quantitatively bound the difference between the 3d inverse Laplacian $(\slap)^{-1}$ and the 2d inverse Laplacian $\Delta_\xi^{-1}$. Third, we bound the difference between the 2d and 3d linear propagators, namely $e^{\tau\L}$ and $S_0$.

\begin{lemma}\label{lem9}
Under the assumptions of Proposition \ref{prop5} with $m > 2$, the error terms $R(\tau)$ defined in (\ref{R}) satisfy the following decay: 
\begin{equation}
\lim_{\tau \rightarrow -\infty}\|R(\tau)\|_{B_zL^{4/3}(m)} = 0.
\end{equation}
\end{lemma}

\begin{proof}
We split $R$ into three error terms and estimate each separately,
\begin{equation}
R(\tau) = \left(v^\xi - \left(\nabla^\perp_\xi (\slap)^{-1}w^z\right)\right) \cdot \nabla_\xi w^z + v^z \sdz w^z - w\cdot \sdv v^z = r_1 + r_2 + r_3.
\end{equation}
The first term, $r_1$, is roughly the difference between 2d and 3d Biot-Savart laws. The velocity field $v^\xi$ is determined from $w$ via the 3d Biot-Savart law, whereas $\nabla^\perp_\xi(\slap)^{-1}w^z$ is the 2d Biot-Savart law, modulo the difference between $\slap$ and $\Delta$. Thus, using (\ref{biotsavartstructure}), we obtain
\begin{equation}
r_1 = -\sdz(\slap)^{-1}(w^\xi)^\perp \cdot \nabla_\xi w^z.
\end{equation}
We now bound $r_1$ using H\"older's inequality, the boundedness of the Biot-Savart law (\ref{biotsavart}), and the embedding $L^2(m) \embeds L^{4/3}$ for $m > 1$:
\begin{equation}\label{r1}
\begin{aligned}
\|r_1\|_{B_zL^{4/3}(m)} &= \|\sdz(-\slap)^{-1}(w^\xi)^\perp \cdot \nabla_\xi w^z\|_{B_zL^{4/3}(m)}\\
	&\lesssim_m \|\sdz(-\slap)^{-1}(w^\xi)^\perp\|_{B_zL^4}  \|\nabla_\xi w^z\|_{B_zL^2(m)}\\
	&\lesssim_m \|w^\xi\|_{B_zL^{4/3}}\|\nabla_\xi w^z\|_{B_zL^2(m)} \lesssim_m \|w^\xi\|_{B_zL^2(m)}\|\nabla_\xi w^z\|_{B_zL^2(m)}.
\end{aligned}
\end{equation}
Similarly for $r_2$, the additional transport term present in the 3D vorticity equation, we have
\begin{equation}\label{r2}
\begin{aligned}
\|r_2\|_{B_zL^{4/3}(m)} &= \|\sdz w^z \nabla_\xi^\perp \cdot (-\slap)^{-1}(w^\xi)\|_{B_zL^{4/3}(m)}\\
	&\lesssim_m \|\sdz w^z\|_{B_zL^2(m)}\|\nabla_\xi^\perp \cdot (-\slap)^{-1}(w^\xi)\|_{B_zL^4}\\
	&\lesssim_m \|\sdz w^z\|_{B_zL^2(m)}\|w^\xi\|_{B_zL^{4/3}} \lesssim_m \|\sdz w^z\|_{B_zL^2(m)}\|w^\xi\|_{B_zL^2(m)}.
\end{aligned}
\end{equation}
We bound $r_3$, the vortex stretching terms present in the 3D vorticity equation, using H\"older's inequality, the boundedness of the Riesz transforms (\ref{riesz}), and the Sobolev embedding \linebreak$B_z\dot{H}^1(m) \embeds B_zL^4$:
\begin{equation}\label{r3}
\begin{aligned}
\|r_3\|_{B_zL^{4/3}(m)} &= \|w \cdot \sdv v^z\|_{B_zL^{4/3}(m)}\\
	&\lesssim_m \|w\|_{B_zL^2(m)}\|\sdv(\nabla_\xi^\perp \cdot (-\slap)^{-1}w^\xi)\|_{B_zL^4}\\
	&\lesssim_m \|w\|_{B_zL^2(m)}\|w^\xi\|_{B_zL^4} \lesssim_m \|w\|_{B_zL^2(m)}\|\nabla_\xi w^\xi\|_{B_zL^2(m)}.
\end{aligned}
\end{equation}
Combining the estimates (\ref{r1}), (\ref{r2}), and (\ref{r3}) and using Proposition \ref{prop5},
\begin{equation}
\|R(\tau)\|_{B_zL^{4/3}(m)} \lesssim \|w^\xi(\tau)\|_{B_zL^2(m)} + \|\nabla_\xi w^\xi(\tau)\|_{B_zL^2(m)} \rightarrow 0.
\end{equation}
\end{proof}

\begin{lemma}\label{lem10}
Under the assumptions of Proposition \ref{prop5}, with $m > 2$, for any $0 < \delta < 1/2$, we have
\begin{equation}
\|\nabla_\xi^\perp (-\Delta_\xi)^{-1} w^z - \nabla_\xi^\perp (-\slap)^{-1}w^z\|_{B_zL^\infty} \lesssim_{m,\delta} \Lambda^{2\delta}\|\sdz w^z\|_{B_zL^2(m)}^{1-2\delta},
\end{equation}
for each $\tau < \tau^*$.
\end{lemma}

\begin{proof}
As in the proof of \big[Corollary 3.4; \cite{jacob}\big] we obtain the pointwise in frequency estimate,
\begin{equation}
\left\|\left[-\nabla^\perp_\xi(|\zeta|^2e^{\tau} - \Delta_\xi)^{-1}\hat w^z\right] - \left[\nabla^\perp_\xi\Delta_\xi^{-1}\hat w^z)\right]\right\|_{L^\infty_\xi} \lesssim_{m,\delta} e^{\tau/2(1-2\delta)}|\zeta|^{1-2\delta}\|\hat w^z\|_{L^2(m)},
\end{equation}
where $\hat w^z(\zeta)$ denotes the Fourier transform of $w^z$ in the $z$ variable. The desired estimate follows from summing over $\zeta$, applying Holder's inequality (in $\zeta$), and recalling for $\tau < \tau^*$, $\|w^z(\tau)\|_{B_zL^2(m)}\le \Lambda$ by assumption.
\end{proof}

\begin{lemma}\label{lem15}
Fix $m > 2$. For any compact set $K_1 \subset [0,\infty) \times B_zL^2(m)$,
\begin{equation}\label{conv1}
\lim_{\tau_0 \rightarrow -\infty} \sup_{(\tau,f) \in K_1} \|S_0(\tau + \tau_0,\tau_0)f - e^{\tau\L}f\|_{B_zL^2(m)} = 0
\end{equation}
and, similarly, for the spatial divergence,
\begin{equation}\label{conv2}
\lim_{\tau_0 \rightarrow -\infty} \sup_{(\tau,f) \in K_2} \|S_0(\tau + \tau_0,\tau_0)\nabla_\xi \cdot f - e^{\tau\L}\nabla_\xi \cdot f\|_{B_zL^2(m)} = 0.
\end{equation}
where $K_2$ is instead taken as a compact subset of $(0,\infty)\times B_zL^{4/3}(m)^2$.
\end{lemma}

\begin{proof}
Fix $\gamma > 0$, $\tau > 0$, and $w\in \S(\R^2\times \T)$ a Schwartz class function. We begin by noting we have sufficient regularity on $w$ to write $S_0$ as a Duhamel integral of $e^{\tau\L}$. In particular, for any $\tau > \tau_0$,
\begin{equation}
S_0(\tau + \tau_0,\tau_0)w - e^{\tau\L}w = \int_{\tau_0}^{\tau_0 + \tau} e^{(\tau + \tau_0 - s)\L} e^s\partial_{zz}S_0(s, \tau_0)w \ ds.
\end{equation}
Now, we note that $\partial_z$ and $\L + e^{\tau}\partial_{zz}$ commute so that $\partial_z$ commutes with $S_0(\tau,s)$.
Thus, taking norms and using the linear propagator estimates for $S_0$ and $e^{\tau\L}$ in Propositions \ref{prop1} and \ref{prop2}, we obtain the estimate
\begin{equation}
\begin{aligned}
\|S_0(\tau + \tau_0,\tau_0)w - e^{\tau\L}w\|_{B_zL^2(m)} &\lesssim_{m} \left(\int_{\tau_0}^{\tau_0+\tau} e^{\gamma(\tau_0 + \tau - s)}e^se^{\gamma(s-\tau_0)}\ ds\right) \|\partial_{zz}w\|_{B_zL^2(m)}\\
 	&\lesssim_{m,w} e^{\gamma\tau}\left[e^\tau - 1\right]e^{\tau_0}.
\end{aligned}
\end{equation}
This implies (\ref{conv1}) holds uniformly for $\tau$ in a compact subset of $\R$ and $w$ a single Schwartz class function, which are dense in $B_zL^2(m)$. Moreover, Propositions \ref{prop1} and \ref{prop2} give the uniform operator norm bound, 
\begin{equation}
\|S_0(\tau + \tau_0,\tau_0)\|_{B_zL^2(m)\rightarrow B_zL^2(m)} + \|e^{\tau\L}\|_{B_zL^2(m)\rightarrow B_zL^2(m)} \lesssim_m e^{\gamma \tau}.
\end{equation}
Thus, Lemma \ref{uniformlem} implies the convergence (\ref{conv1}).

\vspace{\baselineskip}

Note, the convergence (\ref{conv1}) immediately implies the gradient convergence (\ref{conv2}) holds for a single Schwartz class vector field $f$ and $\tau$ in a compact subset of $\R$. By density of $\S(\R^2\times \T)$ in $B_zL^{4/3}(m)$ and the uniform operator norm bound, 
\begin{equation}
\|S_0(\tau + \tau_0,\tau_0)\div_\xi\|_{B_zL^{4/3}(m)\rightarrow B_zL^2(m)} + \|e^{\tau\L}\div_\xi\|_{B_zL^{4/3}(m)\rightarrow B_zL^2(m)} \lesssim_m \frac{e^{(\gamma - 1/2) \tau}}{a(\tau)^{3/4}},
\end{equation}
we conclude that the convergence holds uniformly on compact subsets of $(0,\infty)\times B_zL^{4/3}(m)^2$.
\end{proof}

\subsection{Step 3: Invariance of the $\alpha$-limit set}


\begin{lemma}\label{lem11}
Suppose $w(\tau)$ is a mild solution to (\ref{mildsvorticity}) satisfying the hypotheses of Proposition \ref{prop5} for some $m > 2$, $\Lambda > 0$, and $\tau^* > -\infty$. Then, for each $2 < m^\prime < m$, the $\alpha$-limit set $\A \subset B_zL^2(m^\prime)$ of $\{w^z(\tau)\}_{\tau < \tau^*}$ is nonempty, and there is a time $T(\A)$ such that for each $0 < \tau < T(\A)$, $\Phi(\tau)\A = \A$ where $\Phi(\tau):B_zL^2(m^\prime)\rightarrow B_zL^2(m^\prime)$ is the flow generated by the 2D vorticity equation (\ref{mild2dform1}).
\end{lemma}

\begin{proof}
First, fix $2 < m^\prime < m$ and note that by the compactness statement from Proposition \ref{prop5}, we have already shown $\A \subset B_zL^2(m^\prime)$ is nonempty and compact. Therefore, by Lemma \ref{flowlem}, there is a time $T(\A) > 0$ such that (\ref{mild2dform1}) generates a well-defined flow $\Phi(\tau)$ on $\A$ for $0 < \tau < T(\A)$.

Second, we prove the inclusion $\Phi(\tau)\A \subset \A$ for any $\tau < T(\A)$. Pick $w^\infty \in \A \subset B_zL^2(m^\prime)$. Then, by the definition of $\A$, there exists a sequence $\tau_n \rightarrow -\infty$ such that $w^z(\tau_n) \rightarrow w^\infty$ in $B_zL^2(m^\prime)$. Define $w^\infty(\tau) = \Phi(\tau)w^\infty$. We claim that for each $0 < \tau < T(\A)$,
$$\lim_{n\rightarrow \infty} \|w^z(\tau + \tau_n) - w^\infty(\tau)\|_{B_zL^2(m^\prime)}.$$
Also, let $0< T_1 < T_2 < T(\A)$ be arbitrary times and fix $\gamma > 0$.

We recall that the functions $w^z(\tau+\tau_n)$ satisfy the following integral equations with initial data at time $t_n$:
\begin{equation}\label{eqn5}
\begin{aligned}
w^z&(\tau + \tau_n) = S_0(\tau+\tau_n,\tau_n)w^z(\tau_n)\\
	 - &\int_{\tau_n}^{\tau+\tau_n} S_0(\tau+\tau_n,s)\left[R + (\nabla_\xi^\perp(\slap)^{-1}w^z - \nabla_\xi^\perp\Delta_\xi^{-1}w^z) \cdot \nabla_\xi w^z + \nabla_\xi^\perp\Delta_\xi^{-1}w^z \cdot \nabla_\xi w^z\right] (s) \ ds\\
\end{aligned}
\end{equation}
where $R(s)$ are the error terms defined in (\ref{R}).

Also, recall $w^\infty(\tau)$ satisfies the integral form of (\ref{svorticity2}), which we write as
\begin{equation}
w^\infty(\tau) = e^{\tau\L}w^\infty - \int_0^\tau e^{(\tau-s)\L} \left[\nabla_\xi^\perp\Delta_\xi^{-1}w^\infty \cdot \nabla_\xi w^\infty\right](s) \ ds.
\end{equation}
Changing variables in (\ref{eqn5}) and using
$$\div_\xi \nabla_\xi^\perp\Delta_\xi^{-1}w^z = \div_\xi \nabla_\xi^\perp\Delta_\xi^{-1}w^\infty = 0,$$
we expand the difference between the two equations as
\begin{equation}
\begin{aligned}
w^z(\tau + \tau_n) -& w^\infty(\tau) = \left[S_0(\tau + \tau_n,\tau_n)w^z(\tau_n) - e^{\tau\L}w^\infty\right]\\
	&- \int_0^{\tau} S_0(\tau+\tau_n,s + \tau_n)\left[R + (\nabla_\xi^\perp(\slap)^{-1}w^z - \nabla^\perp_\xi\Delta_\xi^{-1}w^z) \cdot \nabla_\xi w^z\right](s + \tau_n) \ ds\\
	&+ \int_0^\tau \left(e^{(\tau -s)\L} - S_0(\tau + \tau_n,s + \tau_n)\right)\left[\nabla_\xi \cdot (w^z \nabla^\perp_\xi\Delta_\xi^{-1}w^z)(s + \tau_n)\right] \ ds\\
	&+ \int_0^\tau e^{(\tau -s)\L}\left[\nabla_\xi \cdot (w^\infty \nabla_\xi^\perp\Delta_\xi^{-1}w^\infty(s))\right]\\
	&\qquad\qquad- e^{(\tau-s)\L}\left[\nabla_\xi \cdot (w^z \nabla_\xi^\perp\Delta_\xi^{-1}w^z(s + \tau_n))\right] \ ds\\
	&=: L_n(\tau) + E^1_n(\tau) + E_n^2(\tau) + I_n(\tau),
\end{aligned}
\end{equation}
where $L_n(\tau)$ corresponds to the difference in the linear evolution operators, $E_n^1(\tau)$ and $E_n^2(\tau)$ are the error terms arising from treating the 3d vorticity equation as a perturbation of the 2d vorticity equation, and $I_n(\tau)$ is the core term corresponding to the difference between the 2d transport terms from $w^z$ and $w^\infty$. We bound each term in $B_zL^2(m^\prime)$ individually.

First, we bound the linear contribution $L_n(\tau)$ using the linear propagator estimate in Proposition \ref{prop2}:
\begin{equation}
\begin{aligned}
\|L_n(\tau)\|_{B_zL^2(m^\prime)} &\le \|S_0(\tau + \tau_n,\tau_n)w^z(\tau_n) - S_0(\tau + \tau_n,\tau_n)w^\infty\|_{B_zL^2(m^\prime)}\\
	 &\quad+ \|S_0(\tau + \tau_n,\tau_n)w^\infty - e^{\tau\L}w^\infty\|_{B_zL^2(m^\prime)}\\
	 &\lesssim_{m^\prime} e^{\tau\gamma}\|w^z(\tau_n) -w^\infty\|_{B_zL^2(m^\prime)} + \|S_0(\tau + \tau_n,\tau_n)w^\infty - e^{\tau\L}w^\infty\|_{B_zL^2(m^\prime)}.
\end{aligned}
\end{equation}
The first term decays uniformly for $\tau \in [T_1,T_2]$ by the assumption that $w^z(\tau_n) \rightarrow w^\infty$ in $B_zL^2(m^\prime)$. The second term decays uniformly for $\tau \in [T_1,T_2]$ by Lemma \ref{lem15}. Therefore, $L_n(\tau) \rightarrow 0$ in $B_zL^2(m^\prime)$ uniformly for $\tau\in [T_1,T_2]$.

Second, for $E_n^1$, using the linear propagator estimate in Proposition \ref{prop2} and H\"older's inequality, we have,
\begin{equation}
\begin{aligned}
\|E_n^1(\tau)\|_{B_zL^2(m^\prime)} &\lesssim_{m^\prime} \int_0^\tau \left(\frac{e^{\gamma(\tau-s)}}{a(\tau - s)^{1/4}}\|R(s + \tau_n)\|_{B_zL^{4/3}(m^\prime)}\right) \ ds\\
	&+ \int_0^\tau e^{\gamma(\tau-s)}\left(\|\nabla_\xi^\perp(\slap)^{-1}w^z - \nabla_\xi^\perp\Delta_\xi^{-1}w^z\|_{B_zL^\infty}\|\nabla_\xi w^z\|_{B_zL^2(m^\prime)}\right)(s +\tau_n) \ ds.
\end{aligned}
\end{equation}
Now, since $\|\nabla_\xi w^z(s)\|_{B_zL^2(m^\prime)}\lesssim_{m^\prime,\Lambda} 1$ by Proposition \ref{prop5} and 
$$\lim_{\tau \rightarrow -\infty}\|R(\tau)\|_{B_zL^{4/3}(m^\prime)} = 0 \qquad \text{and}\qquad  \lim_{\tau\rightarrow -\infty}\|\nabla_\xi^\perp(\slap_\tau)^{-1}w^z(\tau) - \nabla_\xi^\perp\Delta_\xi^{-1}w^z(\tau)\|_{B_zL^\infty} = 0$$
by Lemma \ref{lem9} and Lemma \ref{lem10}, $E^1_n(\tau)\rightarrow 0$ in $B_zL^2(m^\prime)$, uniformly for $\tau\in [T_1,T_2]$.

Third, to estimate $E_n^2$,
we will use the strong convergence of the operators given in Lemma \ref{lem15} as we did for $L_n(\tau)$. However, we note we must verify that $\{w^z(s)\nabla_\xi^\perp\Delta_\xi^{-1} w^z(s)\}_{s < \tau^*}$ is precompact in $B_zL^{4/3}(m)^2$. To that end, define the bilinear map 
$B(f,g) = f\nabla_\xi^\perp\Delta_\xi^{-1}g$ for $f,g\in B_zL^2(m^\prime)$. Then, using H\"older's inequality and properties of the 2d Biot-Savart law together with Lemma \ref{extensionlem}, we see there exists $C = C(m^\prime) > 0$ such that 
\begin{equation}\label{bilinear}
\|B(f,g)\|_{B_zL^{4/3}(m^\prime)} \le C\|f\|_{B_zL^2(m^\prime)}\|g\|_{B_zL^2(m^\prime)},
\end{equation}
which implies $B:B_zL^2(m^\prime)^2 \rightarrow B_zL^{4/3}(m^\prime)^2$ is continuous. Thus, since the product of compact sets is compact and the image of compact sets under a continuous map is compact, $\{w^z(s)\nabla_\xi^\perp\Delta_\xi^{-1} w^z(s)\}_{s < \tau^*}$ is precompact in $B_zL^{4/3}(m)^2$. Now, Lemma \ref{lem15} implies the sequence of functions $f_n:\set{0 \le s < \tau < \infty} \rightarrow B_zL^2(m^\prime)$, defined by
$$f_n(s,\tau) = \left(e^{(\tau -s)\L} - S_0(\tau + \tau_n,s + \tau_n)\right)\left[\nabla_\xi \cdot (w^z \nabla^\perp_\xi\Delta_\xi^{-1}w^z)(s + \tau_n)\right],$$
converges to $0$ uniformly on compact subsets of $\set{0 \le s < \tau < \infty}$. Therefore, $E^2_n(\tau)\rightarrow 0$ in $B_zL^2(m^\prime)$ uniformly for $\tau \in [T_1,T_2]$.

Fourth, we estimate $I_n(\tau)$ using the gradient bound on $e^{\tau\L}$ in Proposition \ref{prop1} and (\ref{bilinear}) to obtain
\begin{equation}
\begin{aligned}
\|I_n\|_{B_zL^2(m^\prime)} &\lesssim_{m^\prime} \int_0^\tau \frac{e^{(\gamma - 1/2)(\tau -s)}}{a(\tau - s)^{3/4}}\bigg(\left\|\left(w^\infty(s) - w^z(s + \tau_n)\right)\nabla_\xi^\perp\Delta_\xi^{-1}w^\infty(s)\right\|_{B_zL^{4/3}(m^\prime)} \\
	&\qquad +\left\|w^z(s+\tau_n)\nabla_\xi^\perp\Delta_\xi^{-1}\left(w^\infty(s) - w^z(s + \tau_n)\right)\right\|_{B_zL^{4/3}(m^\prime)}\bigg) \ ds\\
	&\lesssim_{m^\prime,\Lambda} \int_0^\tau \frac{e^{(\gamma - 1/2)(\tau -s)}}{a(\tau - s)^{3/4}}\|w^\infty(s) - w^z(s + \tau_n)\|_{B_zL^2(m^\prime)} \ ds.
\end{aligned}
\end{equation}

Summarizing, we obtain a constant $K = K(m^\prime,\Lambda) > 0$ and nonnegative functions $\epsilon_n(\tau)$ such that $\epsilon_n(\tau)\rightarrow 0$ uniformly for $\tau\in [T_1,T_2]$ and
\begin{equation}
\|w^z(\tau + \tau_n) - w^\infty(\tau)\|_{B_zL^2(m^\prime)} \le \epsilon_n(\tau) + K\int_0^\tau \frac{e^{(\gamma - 1/2)(\tau -s)}}{a(\tau - s)^{3/4}}\|w^z(s + \tau_n) - w^\infty(s)\|_{B_zL^2(m^\prime)}\ ds.
\end{equation}
Therefore, taking $T_2$ sufficiently small so that
$$K\sup_{0 < \tau < T_2}\int_0^\tau \frac{e^{(\gamma - 1/2)(\tau -s)}}{a(\tau - s)^{3/4}} \ ds \le \frac{1}{2},$$
we obtain
$$\sup_{T_1\le \tau \le T_2}\|w^z(\tau+\tau_n) - w^\infty(\tau)\|_{B_zL^2(m^\prime)} \le 2\sup_{T_1\le \tau \le T_2}\epsilon_n(\tau) \rightarrow 0.$$
However, $T_1 > 0$ is still arbitrary, and hence for $\tau\in[0,T_2]$ with $T_2$ sufficiently small depending only on $\Lambda$ and $m^\prime$, $w^\infty(\tau) \in \A$ and $\Phi(\tau)\A \subset \A$ for $0 \le \tau \le T_2$. By the semigroup property, we certainly have $\Phi(\tau)\A \subset \A$ for all $0\le \tau \le T(A)$.

To prove the reverse inclusion, $\A \subset \Phi(\tau)\A$, we follow a similar argument. Suppose $w^z(\tau_n) \rightarrow w^\infty$ and $\tau_n \rightarrow -\infty$. Then, by the precompactness of $\{w^z(s)\}_{s < \tau^*}$, up to a subsequence, there is a $w_\tau^\infty \in \A$ such that $w^z(\tau_n - \tau) \rightarrow w^\infty_\tau$. The above argument implies that $w^z(\tau_n) \rightarrow \Phi(\tau)w^\infty_\tau$. By the uniqueness of limits, we conclude $w^\infty = \Phi(\tau)w^\infty_\tau \in \Phi(\tau)\A$. Therefore, for each $0 \le \tau\le T(\A)$, $\Phi(\tau)\A = \A$ and the proof is complete.
\end{proof}

\subsection{Step 4: Rigidity}

\begin{proof}[Proof of Theorem \ref{thm1}]
By assumption, there exist $\Lambda > 0$, $m > 2$ and $\tau^* > -\infty$ such that
$$\sup_{\tau < \tau^*} \|w\|_{B_zL^2(m)} \le \Lambda,$$
and
$$\sup_{\tau < \tau^*} \|\partial_z w^z\|_{B_zL^2(m)} < \infty.$$
Fix $2 < m^\prime < m$, so that by Proposition \ref{prop5} and Lemma \ref{lem11}, if $\A \subset B_zL^2(m^\prime)$ the $\alpha$-limit set of $\{w^z(\tau)\}_{\tau < \tau^*}$, then $\A$ is compact, nonempty, and $\Phi(s)\A = \A$ for each $0 < s < T(\A)$. We will now show that $\A = \{\alpha G\}$.

Let $\pi_z:B_zL^2(m^\prime) \rightarrow L^2(m^\prime)$ denote the projection of a function onto its $z$-th slice, i.e. 
$$\pi_z(f)(\xi) = f(\xi,z)$$
and let $\A(z) \subset L^2(m^\prime)$ denote $\pi_z(\A)$ onto the $z$-th slice. Note, $\pi_z$ is a well-defined bounded linear operator by the embedding $B_zX \embeds C_z(\T;X)$. Then, by Lemma \ref{flowlem}, $\Phi(\tau)$ commutes with projections in the sense that 
$$\pi_z \circ \Phi(\tau) = \phi(\tau) \circ \pi_z,$$
where $\phi(\tau)$ is the flow generated by (\ref{svorticity2}) on $L^2(m^\prime)$. Thus, by the semigroup property for $\phi$, $\phi(\tau)(\A(z)) = \A(z)$ for each $\tau \ge 0$.
Hence, by Theorem \ref{thm4}, for each $z$, $\A(z)\subset \{\beta G\ | \ \beta\in \R\}$.

Now, if $w^\infty \in \A$, then for every $z$, $w^\infty(\cdot,z)$ has finite $L^2(m^\prime)$ norm and is contained in $\A(z)$. Therefore, $w^\infty(\cdot,z) = \beta(z)G(\cdot)$ for each $z$. Finally, using $\omega(t) \weakstar \alpha\delta_{\{x = 0\}}$ as $t\rightarrow 0^+$, we conclude $\beta(z) = \alpha$ and $\A = \{\alpha G\}$.
By the compactness of $\{w^z(\tau)\}_{\tau < \tau^*}$, this is equivalent to $w^z(\tau) \rightarrow \alpha G$ in $B_zL^2(m^\prime)$.

Combined with $w^\xi \rightarrow 0$ in $B_zL^2(m)$, we have shown
$$\lim_{\tau \rightarrow -\infty} \left\|w - \begin{bmatrix} 0\\ 0\\ \alpha G\end{bmatrix}\right\|_{B_zL^2(m^\prime)} = 0,$$
and therefore $w$ lies within the uniqueness class shown in \cite{jacob}. Thus, there is a unique mild solution satisfying the hypotheses of Theorem \ref{thm1}. Moreover, since the Oseen vortex $\omega^g$ satisfies the assumptions of Theorem \ref{thm1}, we must have $w = w^g$. 
\end{proof}

\begin{remark}
An alternative argument that $\beta(z)$ is constant has been offered by one of the reviewers. In particular, take $\tau_n$ a sequence of times such that $\tau_n \rightarrow -\infty$. Then, using the divergence-free condition on $w^z$,
\begin{equation}
e^{\tau_n/2}\partial_z\int w^z(\tau_n,\xi,z) \ d\xi = -\int \nabla_\xi \cdot w^z(\tau_n,\xi,z) \ d\xi = 0,
\end{equation}
for each $z$ and $n$. Therefore, $\int w^z(\tau_n,\xi,z) \ d\xi$ is constant in $z$ and taking a limit in $n$ yields $\beta(z)$ is constant.
\end{remark}


\subsection{Proof of Corollary}

\begin{proof}[Proof of Corollary \ref{cor1}]
By assumption, there exist $\Lambda > 0$, $m > 2$ and $\tau^* > -\infty$ such that
$$\sup_{\tau < \tau^*} \|w\|_{B_zL^2(m)} \le \Lambda,$$
and
$$\sup_{\tau < \tau^*} \|\partial_z w\|_{B_zL^2(m)} < \infty.$$
By Lemma \ref{lem7}, there is a constant $C > 0$, such that $\|\sdv w^\xi(\tau)\|_{B_zL^2(m)} \le C(\Lambda,m)$. Therefore, fixing $2 < m^\prime < m$, by the same argument as in Lemma \ref{lem8}, using $\|\partial_z w^\xi\|_{B_zL^2(m)} < \infty$, we conclude that $\{w^\xi(\tau)\}_{\tau < \tau^*}$ is precompact as a subset of $B_zL^2(m^\prime)$. Taking any sequence $w^\xi(\tau_n)$, there is a further sequence which converges strongly in $B_zL^2(m^\prime)$ and weakly to $0$ in $B_zL^2(m)$. By weak-strong uniqueness of limits, $w^\xi(\tau_n) \rightarrow 0$ in $B_zL^2(m^\prime)$. Thus, $$\lim_{\tau \rightarrow -\infty} \|w^\xi(\tau)\|_{B_zL^2(m^\prime)} = 0.$$ Applying Theorem \ref{thm1} completes the proof.
\end{proof}


\section{Proof of Theorem 2}

\subsection{Proof Sketch}

There are a few modifications we need to make to the proof of Theorem \ref{thm1}. We expect our solution to converge to $w^g = (0,0,\alpha G)^t$. However, $w^g$ does not decay at infinity in the $z$-direction. In fact, since $w^g$ is $z$-independent, the Fourier transform of $w^g$ in the $z$-direction is a Dirac mass centered at the origin. While this has finite total variation norm, it does not have finite $L^1$ norm and $w^g$ is not contained in $B_zL^2(m)$. 
Therefore, we analyze solutions for which $w - w^g \in B_zL^2(m)$ and treat the vorticity equation for $w$ as a perturbation of the vorticity equation linearized about $w^g$.

Taking $\omega:\R^3\times (0,\infty) \rightarrow \R^3$ a mild solution with circulation parameter $\alpha$, we define the core correction piece $\omega_c = \omega - \omega^g$. This induces a corresponding decomposition in self-similar coordinates,
$$w = w_c + w^g, \qquad v = v_c + v^g$$
where $w_c$ satisfies the vector-valued PDE
\begin{equation}\label{core}
\begin{aligned}
\partial_\tau w^\xi_c + \alpha (v^g \cdot \nabla_\xi)w^\xi_c - \alpha(w^\xi_c \cdot \nabla_\xi)v^g- (\L + \sdz^2)w^\xi_c  &= - (v_c \cdot \sdv) w^\xi_c + (w_c \cdot \sdv)v^\xi_c\\
\partial_\tau w^z_c + \alpha (v^g \cdot\nabla_\xi) w^z_c + \alpha \nabla_\xi G \cdot v^\xi_c - \alpha G\sdz v^z_c - (\L + \sdz^2)w^z_c &= -(v_c \cdot \sdv) w^z_c + (w_c \cdot \sdv) v^z_c. 
\end{aligned}
\end{equation}
We have placed terms we wish to treat perturbatively on the right hand side. The left hand side of the above equations is exactly the self-similar vorticity equations linearized about $\alpha G$. The linearized equations generate a strongly continuous flow $S_\alpha(\tau,\tau_0)$ on $B_zL^2(m)^3$ by Proposition \ref{prop2}. Thus, we can rewrite (\ref{core}) as the following integral equation: for each $0\le \tau_0 < \tau$,
\begin{equation}\label{coremild}
w_c(\tau) = S_\alpha(\tau,\tau_0)w_c(\tau_0) - \int_{\tau_0}^\tau S_\alpha(\tau,s)\left[v_c \cdot \sdv_s w_c - w_c \cdot \sdv_s v_c\right](s) \ ds
\end{equation}
where the integral converges in $B_zL^2(m)$. In particular, we obtain the following equation for $w^z_c$:
\begin{equation}\label{corezmild}
w_c^z(\tau) = S_\alpha^z(\tau,\tau_0)w_c(\tau_0) - \int_{\tau_0}^\tau S_\alpha^z(\tau,s)\left[v_c \cdot \sdv_s w_c - w_c \cdot \sdv_s v_c\right](s) \ ds,
\end{equation}
where, because $S_\alpha$ is not diagonal, we do not obtain an equation involving just terms from the differential form of the $w^z_c$ equation. Now, formally throwing away terms in the $S_\alpha$ equation with a $\sdz$ and replacing $\slap$ with $\Delta_\xi$, we expect $S_\alpha^z f \rightarrow \Tau_\alpha f^z$, where $\Tau_\alpha$ is the semigroup generated by the 2d vorticity equation linearized about $\alpha G$, namely,
\begin{equation}
\partial_\tau f + \alpha v^g\cdot \nabla_\xi f + \alpha (\nabla_\xi^\perp\Delta_\xi^{-1}f) \cdot \nabla_\xi G - \L f = 0.
\end{equation}
Therefore, replacing $S_\alpha^z$ by $\Tau_\alpha$ in (\ref{corezmild}) and separating out two types of error terms, we obtain an equation analogous to (\ref{heuristiceqn}):
\begin{equation}
w_c^z(\tau) = \Tau_\alpha(\tau - \tau_0)w_c^z(\tau_0) - \int_{\tau_0}^\tau \Tau_\alpha(\tau-s)\left[\left(\nabla_\xi^\perp(\slap)^{-1} w_c^z\cdot \nabla_\xi w_c^z\right) + R^\prime\right](s) \ ds + \mathcal{E}(\tau,\tau_0),
\end{equation}
where $\mathcal{E}(\tau,\tau_0)$ is the error term due to replacing $S_\alpha^z$ with $\Tau_\alpha$ and $R^\prime$ is the error term analagous to $R$, defined in (\ref{R}), given by
\begin{equation}
R^\prime(s) = \div_\xi(v^\xi_c w_c^z - w_c^\xi v^z_c) - \nabla_\xi^\perp(\slap)^{-1}w_c^z \cdot \nabla_\xi w_c^z.
\end{equation}
We will show both $\mathcal{E}$ and $R^\prime$ converge to $0$ in some appropriate sense and justify that $w_c^z$ asymptotically solves the 2d vorticity equation in the form
\begin{equation}\label{mildlimit}
w^*(\tau) = \Tau_\alpha(\tau-\tau_0)w^*(\tau_0) - \int_{\tau_0}^\tau \Tau_\alpha(\tau-s)[\nabla_\xi^\perp \Delta_\xi^{-1}w^*(s) \cdot \nabla_\xi w^*(s)]\ ds.
\end{equation}
This will allow us to apply a similar rigidity argument as in the proof of Theorem \ref{thm1}.
By parabolic regularity, $\|w_c(\tau)\|_{B_zL^2(m)} \lesssim 1$ is enough to justify the following lemma:
\begin{lemma}\label{mildlemma}
Let $\omega$ be as in the statement of Theorem \ref{thm2}. Then, $w_c$ defined by $w - w^g$ satisfies (\ref{coremild}).
\end{lemma}


Now, we outline the proof of Theorem \ref{thm2}. The proof of Theorem \ref{thm2} will follow the same structure as the proof of Theorem \ref{thm1}, but with additional difficulties stemming from linearizing about $\alpha G$. We proceed in four steps:
\begin{itemize}
\item First, we will analyze mild solutions to (\ref{coremild}) and use parabolic regularity theory to show that the trajectory $\{w^z_c(\tau)\}_{\tau < \tau^*}$ is precompact in $B_zL^2(m)$. This will imply that the $\alpha$-limit set $\A\subset B_zL^2(m)$ is nonempty. The main differences with the analogous step in the proof of Theorem \ref{thm1} are in the proof that $\sdv w^\xi \rightarrow 0$ and in showing compactness of $\{w^z\}_{\tau < \tau^*}$.
\item Second, we will show that the remainder terms $R^\prime(\tau)$ vanish as $\tau \rightarrow -\infty$. This step is nearly identical to the analogous step in the proof of Theorem \ref{thm1}.
\item Third, we will use Steps 1 and 2 to show that the $\alpha$-limit set $\A$ is invariant under the flow $\Phi^\prime(\tau)$ generated by the 2d vorticity equation in the form (\ref{mildlimit}).
\item Fourth, we will show that this implies $\A$ is invariant under the flow $\Phi(\tau)$ generated by the 2d vorticity equation (\ref{svorticity2}) and conclude by the rigidity of invariant sets of $\Phi$ that $\A = \{0\}$.
\end{itemize}

\subsection{Semigroup Asymptotics}

Before beginning to analyze the regularity properties of mild solutions to (\ref{coremild}), we study the asymptotic behavior of $S_\alpha$. Throughout the proof of Theorem \ref{thm1}, we studied the equations for $w^z$ and $w^\xi$ separately and deduced properties of mild solutions based on estimates of the form
\begin{equation}\label{diagest}
\|S_0(\tau,\tau_0)w^\xi\|_{B_zL^2(m)} \lesssim e^{\gamma \tau}\|w^\xi\|_{B_zL^2(m)},
\end{equation}
where only $w^\xi$ appears on the right hand side. However, this type of estimate is not true for $S_\alpha$, as $S_\alpha$ is not block diagonal in the sense that
\begin{equation}\label{failure}
S_\alpha w_c \neq \begin{bmatrix}
S_\alpha^\xi & 0\\
0 & S_\alpha^z
\end{bmatrix}
\begin{bmatrix} w^\xi_c\\ w^z_c\end{bmatrix}.
\end{equation}
We only need estimates of the form (\ref{diagest}) when we work in the asymptotic regime $\tau \rightarrow -\infty$. The purpose of this section is to replace the failure in (\ref{failure}) with asymptotic statements that are all variations on
\begin{equation}
S_\alpha w \rightarrow
\begin{bmatrix}
\Gamma_\alpha & 0\\
0 & \Tau_\alpha
\end{bmatrix}
\begin{bmatrix} w^\xi_c\\ w^z_c\end{bmatrix}.
\end{equation}

We first need the following lemma to control the Biot-Savart law.

\begin{lemma}\label{lem14}
For $m > 2$, $f\in B_zL^2(m)$, and $0 < \delta < 1/2$, we have the following estimates:
\begin{equation}
\begin{aligned}
\|\nabla_\xi(\Delta_\xi^{-1} - (\slap)^{-1})f\|_{B_zL^\infty} + \|\sdz (\slap)^{-1} f\|_{B_zL^\infty} &+ \|\sdz\nabla_\xi(\slap)^{-1}\|_{B_zL^2}\\
	&\lesssim_{\delta,m} \|\sdz f\|_{B_zL^2(m)}^{1-2\delta}\|f\|_{B_zL^2(m)}^{2\delta}.
\end{aligned}
\end{equation}
\end{lemma}

\begin{proof}
We take the Fourier transform in $z$ and as in the proof of \big[Lemma 3.3 and Corollary 3.4; \cite{jacob}\big], we obtain the pointwise in frequency estimate
\begin{equation}
\begin{aligned}
\|\nabla_\xi(\Delta_{\xi}^{-1} - &(e^\tau|\zeta|^2 - \Delta_{\xi})^{-1})\hat f\|_{L^\infty_\xi} + e^{\tau/2}|\zeta|\|(e^{\tau}|\zeta|^2 - \Delta_{\xi})^{-1}\hat f\|_{L^\infty_\xi}\\
	&+ e^{\tau/2}|\zeta|\|\nabla_\xi(e^\tau|\zeta|^2 - \Delta_\xi)^{-1}\hat f\|_{L^2_\xi} \lesssim_{\delta,m}
e^{(\tau/2)(1-2\delta)}|\zeta|^{1-2\delta}\|\hat{f}\|_{L^2(m)}.
\end{aligned}
\end{equation}
Integrating in $\zeta$ and using H\"older's inequality (in $\zeta$), we obtain the desired estimate.
\end{proof}

\begin{lemma}\label{lem16}
Fix $m > 2$. For any compact set $K_1 \subset [0,\infty) \times B_zL^2(m)^3$,
\begin{equation}\label{conv3}
\lim_{\tau_0 \rightarrow -\infty} \sup_{(\tau,f) \in K_1} \|S_\alpha(\tau + \tau_0,\tau_0)f - (\Gamma_\alpha(\tau)f^\xi,\Tau_\alpha(\tau)f^z)\|_{B_zL^2(m)} = 0.
\end{equation}
\end{lemma}

\begin{proof}
The proof is analogous to that of (\ref{conv2}) in Lemma \ref{lem15}, but with more complicated error terms. To simplify the argument, we prove (\ref{conv3}) in two parts: first, we show $S_\alpha^\xi \rightarrow \Gamma_\alpha$ and, second, we show $S_\alpha^z \rightarrow \Tau_\alpha$. Fix $\gamma > 0$, $\tau > 0$, and $w\in\S(\R^3)^3$. Thus, $w$ is sufficiently regular to justify writing $S_\alpha^\xi$ as a Duhamel integral of  $\Gamma_\alpha$ using (\ref{Salphadefn}):
\begin{equation}\label{Duhamel}
\begin{aligned}
S_\alpha^\xi(\tau + \tau_0,\tau_0)w - \Gamma_\alpha&(\tau+\tau_0,\tau_0)w^\xi = \\
	-&\int_{\tau_0}^{\tau_0 + \tau} \Gamma_\alpha(\tau_0 + \tau - s)\left[\alpha G e^s\partial_z^2(\slap)^{-1}(S_\alpha^\xi(s,\tau_0)w)^\perp\right] \ ds\\
	+&\int_{\tau_0}^{\tau_0 + \tau} \Gamma_\alpha(\tau_0 + \tau - s)\left[ \alpha G e^{s/2}\partial_z \nabla_\xi^\perp(\slap)^{-1}(S_\alpha^z(s,\tau_0)w)\right] \ ds\\
	+&\int_{\tau_0}^{\tau_0 + \tau} \Gamma_\alpha(\tau_0 + \tau - s)\left[e^s\partial_z^2(S_\alpha^\xi(s,\tau_0)w)\right] \ ds\\
	=:&E_1 + E_2 + E_3.
\end{aligned}
\end{equation}
The terms $E_j$ for $1\le j \le 3$ are error terms which we will now bound individually. We bound $E_1$ using the following estimate from Lemma \ref{lem14} with $\delta = 1/4$:
\begin{equation}\label{bound4}
\|\partial_z (\slap_s)^{-1}f\|_{B_zL^\infty_\xi} \lesssim_m e^{-s/4}\|\partial_z f\|_{B_zL^2(m)}^{1/2}\|f\|^{1/2}_{B_zL^2(m)}.
\end{equation}
Using the linear propagator bounds in Proposition \ref{prop1}, the $z$-independence and rapid decay of $G$, (\ref{bound4}), $S_\alpha$ and $\partial_z$ commute, and the linear propagator bounds in Proposition \ref{prop2}, we obtain
\begin{equation}
\begin{aligned}
\|E_1\|_{B_zL^2(m)} &\lesssim_m \int_{\tau_0}^{\tau_0 + \tau} e^{\gamma(\tau_0 + \tau - s)}e^s \|\alpha G\partial_z^2(\slap)^{-1}(S^\xi_\alpha(s,\tau_0)w)^\perp\|_{B_zL^2(m)} \ ds\\
	&\lesssim_{m,\alpha} \int_{\tau_0}^{\tau_0 + \tau} e^{\gamma(\tau_0 + \tau - s)}e^s \|G(\xi)\langle \xi \rangle^{2m}\|_{L^2_\xi}\|\partial_z^2(\slap)^{-1}(S^\xi_\alpha(s,\tau_0)w)^\perp\|_{B_zL^\infty} \ ds\\
	&\lesssim_{m,\alpha} \int_{\tau_0}^{\tau_0 + \tau} e^{\gamma(\tau_0 + \tau - s)}e^se^{-s/4}\|S^\xi_\alpha(s,\tau_0)\partial_z^2w\|_{B_zL^2(m)}^{1/2}\|S_\alpha^\xi\partial_zw\|_{B_zL^2(m)}^{1/2} \ ds\\
	&\lesssim_{m,\alpha} \left(\int_{\tau_0}^{\tau_0 + \tau} e^{\gamma\tau}e^{(3/4)s} \ ds\right) \|\partial_z^2w\|_{B_zL^2(m)}^{1/2}\|\partial_zw\|_{B_zL^2(m)}^{1/2}\\
	&\lesssim_{m,\alpha,w} e^{\gamma\tau}\left[e^{3/4\tau} - 1\right]e^{(3/4)\tau_0}.\\
\end{aligned}
\end{equation}
We estimate $E_2$ similarly, only replacing (\ref{bound4}) with the following estimate from Lemma \ref{lem14} with $\delta = 1/4$:
\begin{equation}\label{bound5}
\|\partial_z \nabla_\xi^\perp(\slap_s)^{-1}f\|_{B_zL^2} \lesssim_{m} e^{-s/4}\|\partial_z f\|_{B_zL^2(m)}^{1/2}\|f\|_{B_zL^2(m)}^{1/2}.
\end{equation}
Estimating as for $E_1$, we obtain
\begin{equation}
\begin{aligned}
\|E_2\|_{B_zL^2(m)} &\lesssim_{m,\alpha} \left(\int_{\tau_0}^{\tau_0 + \tau} e^{\gamma(\tau_0 + \tau - s)}e^{s/2}e^{-s/4}e^{\gamma(s - \tau_0)} \ ds \right)\|\partial_z w\|_{B_zL^2(m)}^{1/2}\|w\|_{B_zL^2(m)}^{1/2}\\
	&\lesssim_{m,\alpha,w} e^{\gamma\tau}\left[e^{\tau/4} - 1\right]e^{\tau_0/4} .
\end{aligned}
\end{equation}
Finally, we estimate $E_3$ as in the proof of (\ref{conv1}) to obtain
\begin{equation}
\|E_3\|_{B_zL^2(m)} \lesssim_{m,\alpha,w} e^{\gamma\tau}\left[e^\tau - 1\right]e^{\tau_0}.
\end{equation}
To summarize, we obtain
\begin{equation}
\|S_\alpha^\xi(\tau_0 + \tau,\tau_0)w - \Gamma_\alpha(\tau)w^\xi\|_{B_zL^2(m)} \lesssim_{m,\alpha,w} e^{(1 + \gamma)\tau}e^{\tau_0}.
\end{equation}
Therefore, (\ref{conv3}) holds for the $\xi$-component for a single Schwartz class vector field $w$ and $\tau > 0$ in a compact subset of $\R$. By Propositions \ref{prop1} and \ref{prop2}, we also have the uniform operator norm bound,
\begin{equation}
\|S_\alpha(\tau + \tau_0,\tau_0)\|_{B_zL^2(m)\rightarrow B_zL^2(m)} + \|\Gamma_{\alpha}(\tau)\|_{B_zL^2(m)\rightarrow B_zL^2(m)} \lesssim_{m,\alpha} e^{\gamma \tau}.
\end{equation}
We conclude (\ref{conv3}) holds for the $\xi$-component by Lemma \ref{uniformlem}.

\vspace{\baselineskip}

We now prove $S_\alpha^z \rightarrow \Tau_\alpha$ in the sense of (\ref{conv3}). Again, fix $\gamma > 0$, $\tau > 0$, $w\in \S(\R^3)^3$, and let $\tau_0 \in \R$. Thus, $w$ is sufficiently regular to justify writing $S_\alpha^z$ as Duhamel integral of  $\Tau_\alpha$ using (\ref{Salphadefn}):
\begin{equation}\label{Duhamel2}
\begin{aligned}
S_\alpha^z(\tau + \tau_0,\tau_0)w - \Tau_\alpha&(\tau+\tau_0-\tau_0)w^z = \\
	&\int_{\tau_0}^{\tau_0 + \tau} \Tau_\alpha(\tau_0 + \tau-s)\left[\alpha \nabla^\perp_\xi(\Delta_\xi^{-1} - (\slap)^{-1})S_\alpha^z(s,\tau_0)w\right] \cdot \nabla_\xi G \ ds\\
	+&\int_{\tau_0}^{\tau_0 + \tau} \Tau_\alpha(\tau_0 + \tau-s)\left[ \alpha\sdz (\slap)^{-1}(S_\alpha^\xi(s,\tau_0)w)^\perp\right]\cdot \nabla_\xi G\ ds\\
	-&\int_{\tau_0}^{\tau_0 + \tau} \Tau_\alpha(\tau_0 + \tau-s)\left[\alpha G\sdz \nabla^\perp_\xi \cdot (\slap)^{-1}S_\alpha^\xi(s,\tau_0)w\right] \ ds\\
	+&\int_{\tau_0}^{\tau_0 + \tau} \Tau_\alpha(\tau_0 + \tau-s)\left[e^s\partial_z^2S_\alpha^z(s,\tau_0)w\right] \ ds\\
	=:&I_1 + I_2 + I_3 + I_4.
\end{aligned}
\end{equation}
The terms $I_j$ for $1\le j \le 4$ are error terms which we will now bound individually. For $I_1$, we expect smallness for $\tau_0 \ll -1$ to come from the difference $\slap - \Delta_\xi$. In particular, we use the following estimate from Lemma \ref{lem14} with $\delta = 1/4$:
\begin{equation}\label{bound3}
\|\nabla^\perp_\xi\left[(\slap)^{-1}_s - \Delta_\xi^{-1}\right]f\|_{B_zL^\infty_\xi} \lesssim_{m} e^{s/4}\|\partial_z f\|_{B_zL^2(m)}^{1/2}\|f\|_{B_zL^2(m)}^{1/2}.
\end{equation}
Using the linear propagator bounds in Proposition \ref{prop1}, the $z$-independence and rapid decay of $\nabla G$, (\ref{bound3}), $S_\alpha$ and $\partial_z$ commute (the variable coefficients in (\ref{Salphadefn}) are $z$-independent), and the linear propagator bounds in Proposition \ref{prop2}, we obtain
\begin{equation}
\begin{aligned}
\| I_1\|_{B_zL^2(m)} &\lesssim_{m} \int_{\tau_0}^{\tau_0 + \tau} e^{\gamma(\tau_0 + \tau - s)}\left\|\left[\alpha \nabla^\perp_\xi((\slap)^{-1} - \Delta_\xi^{-1})S_\alpha^z(s,\tau_0)w\right] \cdot \nabla_\xi G\right\|_{B_zL^2(m)} \ ds\\
	&\lesssim_{m,\alpha} \int_{\tau_0}^{\tau_0 + \tau} e^{\gamma(\tau_0 + \tau - s)}\|\nabla_\xi G(\xi)\langle\xi\rangle^{2m}\|_{L_\xi^2}\left\|\nabla_\xi^\perp(\slap^{-1} - \Delta^{-1}_\xi)S_\alpha^z(s,\tau_0)w\right\|_{B_zL^\infty} \ ds\\
	&\lesssim_{m,\alpha} \int_{\tau_0}^{\tau_0 + \tau}e^{\gamma(\tau_0 + \tau - s)}e^{s/4}\|\partial_z S_\alpha^z(s,\tau_0)w\|_{B_zL^2(m)}^{1/2}\|S_\alpha^z(s,\tau_0)w\|_{B_zL^2(m)}^{1/2} \ ds\\
	&\lesssim_{m,\alpha} \left(\int_{\tau_0}^{\tau_0 + \tau}e^{\gamma(\tau_0 + \tau -s)}e^{s/4}e^{\gamma(s-\tau_0)} \ ds\right) \|\partial_zw\|_{B_zL^2(m)}^{1/2}\|w\|_{B_zL^2(m)}^{1/2}\\
	&\lesssim_{\alpha,m,w} e^{\gamma\tau}\left[e^{\tau/4} - 1\right]e^{\tau_0/4}.
\end{aligned}
\end{equation}
We estimate similarly for $I_2$, where we replace (\ref{bound3}) with the estimate (\ref{bound4}) from Lemma \ref{lem14} with $\delta = 1/4$:
\begin{equation}
\begin{aligned}
\|I_2\|_{B_zL^2(m)} &\lesssim_{m,\alpha} \left(\int_{\tau_0}^{\tau_0 + \tau}e^{\gamma(\tau_0 + \tau -s)}e^{s/2}e^{-s/4}e^{\gamma(s-\tau_0)} \ ds\right) \|\partial_zw\|_{B_zL^2(m)}^{1/2}\|w\|_{B_zL^2(m)}^{1/2}\\
	&\lesssim_{\alpha,m,w} e^{\gamma\tau}\left[e^{\tau/4} - 1\right]e^{\tau_0/4}.
\end{aligned}
\end{equation}
The estimate for $I_3$ is again similar, where the replacement for (\ref{bound3}) is the estimate (\ref{bound5}), also from Lemma \ref{lem14} with $\delta = 1/4$:
\begin{equation}
\begin{aligned}
\| I_3\|_{B_zL^2(m)} &\lesssim_{m,\alpha} \int_{\tau_0}^{\tau_0 + \tau} e^{\gamma(\tau_0 + \tau - s)}\|G(\xi)\langle\xi\rangle^{2m}\|_{L_\xi^\infty}e^{s/2}\left\|\partial_z\nabla^\perp_\xi\cdot (\slap)^{-1}S_\alpha^\xi(s,\tau_0)w\right\|_{B_zL^2} \ ds\\
	&\lesssim_{m,\alpha} \left(\int_{\tau_0}^{\tau_0 + \tau}e^{\gamma(\tau_0 + \tau -s)}e^{s/2}e^{-s/4}e^{\gamma(s-\tau_0)} \ ds\right) \|\partial_zw\|_{B_zL^2(m)}^{1/2}\|w\|_{B_zL^2(m)}^{1/2}\\
	&\lesssim_{\alpha,m,w} e^{\gamma\tau}\left[e^{\tau/4} - 1\right]e^{\tau_0/4}.
\end{aligned}
\end{equation}
Finally, we estimate $I_4$ as in the proof of (\ref{conv1}), to obtain
\begin{equation}
\| I_4\|_{B_zL^2(m)} \lesssim_{m,\alpha,w} e^{\gamma\tau}\left[e^\tau - 1\right]e^{\tau_0}.
\end{equation}
Therefore, we have shown
\begin{equation}
\|S_\alpha^z(\tau + \tau_0,\tau_0)w - \Tau_\alpha(\tau)w^z\|_{B_zL^2(m)} \lesssim_{m,\alpha,w} e^{(1 + \gamma)\tau}e^{\tau_0}.
\end{equation}
This proves (\ref{conv3}) for a single Schwartz class function $w$ and times $\tau$ in a compact subset of $\R$. The full statement of (\ref{conv3}) follows from noting that by Propositions \ref{prop1} and \ref{prop2}, we have the uniform bound
\begin{equation}
\|S_\alpha(\tau + \tau_0,\tau_0)\|_{B_zL^2(m)\rightarrow B_zL^2(m)} + \|\Tau_{\alpha}(\tau)\|_{B_zL^2(m)\rightarrow B_zL^2(m)} \lesssim_{m,\alpha} e^{\gamma \tau}.
\end{equation}
Thus, Lemma \ref{uniformlem} implies (\ref{conv3}).
\end{proof}

\begin{lemma}\label{lem17}
Fix $m > 2$. For $K_2$ any compact subset of $(0,\infty) \times B_zL^2(m)^3$,
\begin{equation}\label{conv4}
\begin{aligned}
\lim_{\tau_0 \rightarrow -\infty} \sup_{(\tau,f)\in K_2} &\biggr(\left\|\nabla_\xi S_\alpha^\xi(\tau + \tau_0, \tau_0)f - \nabla_\xi \Gamma_\alpha(\tau)f^\xi\right\|_{B_zL^2(m)}\\ 
 &+ \left\|\nabla_\xi S_\alpha^z(\tau + \tau_0,\tau_0)f - \nabla_\xi \Tau_\alpha(\tau)f^z\right\|_{B_zL^2(m)}\\
	&+ e^{(\tau+\tau_0)/2}\|\partial_z S_\alpha(\tau+\tau_0,\tau_0)f\|_{B_zL^2(m)}\biggr)  = 0.
\end{aligned}
\end{equation}
Moreover, let $F: (-\infty,\tau^*) \rightarrow M_{3,3}(\R)$ be such that $\divz_\tau\divz_\tau F(\tau) = 0$ for each $\tau < \tau^*$ and such that for each $1\le i,j\le 3$, $\{F_{ij}(s)\}_{s < \tau^*}$ is a precompact subset of $B_zL^{4/3}(m)$. Then, for any $K_3\subset (0,\infty)$ compact, 
\begin{equation}\label{conv5}
\begin{aligned}
\lim_{\tau_0 \rightarrow -\infty} \sup_{\tau \in K_3} &\biggr(\left\|S_\alpha^\xi(\tau + \tau_0,\tau_0)\divz_{\tau_0}F(\tau_0) - \Gamma_\alpha(\tau)(\div_\xi F^{\xi,\xi}(\tau_0))\right\|_{B_zL^2(m)}\\ 
	&+ \left\|S_\alpha^z(\tau + \tau_0,\tau_0)\divz_{\tau_0}F(\tau_0) -\Tau_\alpha(\tau)(\div_\xi F^{z,\xi}(\tau_0))\right\|_{B_zL^2(m)}\biggr) = 0,
\end{aligned}
\end{equation}
where we write $F$ in block matrix form as
\begin{equation}F = \begin{bmatrix}\label{block}
F^{\xi,\xi} & F^{\xi,z}\\
F^{z,\xi} & F^{z,z}
\end{bmatrix}.\end{equation}
Furthermore, for each such $K_3$,
\begin{equation}\label{conv6}
\begin{aligned}
\lim_{\tau_0 \rightarrow -\infty} \sup_{\tau\in K_3} &\biggr(\left\|\nabla_\xi S_\alpha^\xi(\tau + \tau_0,\tau_0)\divz_{\tau_0}F(\tau_0) - \nabla_\xi\Gamma_\alpha(\tau)(\div_\xi F^{\xi,\xi}(\tau_0))\right\|_{B_zL^2(m)}\\ 
	&+ \left\|\nabla_\xi S_\alpha^z(\tau + \tau_0,\tau_0)\divz_{\tau_0}F(\tau_0) -\nabla_\xi\Tau_\alpha(\tau)(\div_\xi F^{z,\xi}(\tau_0))\right\|_{B_zL^2(m)}\\
&+e^{(\tau +\tau_0)/2}\|\partial_zS_\alpha(\tau + \tau_0,\tau_0)\divz_{\tau_0}F(\tau_0)\|_{B_zL^2(m)}\biggr) = 0,
\end{aligned}
\end{equation}
with the same notation.
\end{lemma}

\begin{proof}
First, the gradient convergence (\ref{conv4}) for a single function $w\in \S(\R^3)^3$ and time $\tau > 0$ follows in the same manner as (\ref{conv3}). Indeed, we write $\nabla_\xi S^\xi_\alpha$ as a Duhamel integral of $\nabla_\xi \Gamma_\alpha$ by taking a gradient of (\ref{Duhamel}) to obtain
$$\nabla_\xi S^\xi_\alpha(\tau + \tau_0, \tau_0) - \nabla_\xi\Gamma_\alpha(\tau) = \sum_{i = 1}^3 \nabla_\xi E_i,$$
where the error terms $E_i$ are those defined in (\ref{Duhamel}). We now estimate $E_i$ as in the preceding lemma, but replacing the linear propagator bounds on $\Gamma_\alpha$ with those on $\nabla_\xi\Gamma_\alpha$. For example, we estimate $E_1$
using the gradient bounds in Proposition \ref{prop1}, the $z$-independence and rapid decay of $G$, (\ref{bound4}), $S_\alpha$ and $\partial_z$ commute, the linear propagator bounds in Proposition \ref{prop2}, and a change of variables to obtain
\begin{equation}
\begin{aligned}
\|\nabla_\xi E_1\|_{B_zL^2(m)} &\lesssim_{m,\alpha,w} \int_{\tau_0}^{\tau_0 + \tau} \frac{e^{\gamma(\tau_0 + \tau - s)}}{a(\tau_0 + \tau - s)^{1/2}} e^s e^{-s/4} e^{\gamma(s - \tau_0)} \ ds\\
	&\lesssim_{m,\alpha,w} e^{(3/4)\tau_0}\int_0^\tau \frac{e^{(3/4)s + \gamma \tau}}{a(\tau - s)^{1/2}} \ ds\\
	&\lesssim_{m,\alpha,w} (\tau + \tau^{1/2})e^{(\gamma + 3/4)\tau}e^{(3/4)\tau_0}.
\end{aligned}
\end{equation}
The estimates of the other error terms $E_2$ and $E_3$ may be modified analogously. Therefore, the $\xi$-component of (\ref{conv4}) holds for $\tau$ in a compact subset of $\R$ and $w$ a Schwartz class vector field. The full statement follows from the operator norm bound from Propositions \ref{prop1} and \ref{prop2},
\begin{equation}
\|\nabla_\xi S_\alpha(\tau + \tau_0,\tau_0)\|_{B_zL^2(m)\rightarrow B_zL^2(m)} + \|\nabla_\xi \Gamma_{\alpha}(\tau)\|_{B_zL^2(m)\rightarrow B_zL^2(m)} \lesssim_{m,\alpha} \frac{e^{\gamma \tau}}{a(\tau)^{1/2}},
\end{equation}
and Lemma \ref{uniformlem}. The $z$-component of (\ref{conv4}) follows from a similar modification of the estimates in Lemma \ref{lem16}. The $\sdz$ convergence in (\ref{conv4}) follows by commuting $\partial_z$ and $S_\alpha$ for a Schwartz class vector field and appealing to Lemma \ref{uniformlem}.

Now, we will show (\ref{conv5}) for a single time $\tau > 0$ and a Schwartz class matrix $W\in M_3(\R^3)$, i.e. $W_{ij} \in S(\R^3)$ for each $1\le i,j \le 3$. Fix such a $\tau > 0$, $W$, and a $\gamma > 0$. Then, using the block matrix notation introduced in (\ref{block}), split $\divz W$ as
$$\divz_{\tau_0} W = \begin{bmatrix}
\div_\xi W^{\xi,\xi}\\
\div_\xi W^{z,\xi}
\end{bmatrix} +
e^{\tau_0/2}\partial_z\begin{bmatrix}
W^{\xi,z}\\
W^{z,z}
\end{bmatrix} := W_1 + e^{\tau_0/2}\partial_z W_2.$$
Thus, using the estimates in Proposition \ref{prop2}, we obtain
\begin{equation}
\|S_\alpha(\tau + \tau_0,\tau_0)e^{\tau_0/2}\partial_zW_2\|_{B_zL^2(m)} \lesssim_{m,\alpha} e^{\gamma \tau + \tau_0/2}\|\partial_z W_2\|_{B_zL^2(m)},
\end{equation}
which decays as $\tau_0\rightarrow -\infty$. Combined with Lemma \ref{lem16} applied to $W_1$, this proves (\ref{conv5}) applied to a single Schwartz class matrix $W$ and time $\tau$ in a compact subset of $\R$. The full claim then follows by the operator norm bounds in Proposition \ref{prop1} and \ref{prop2} and Lemma \ref{uniformlem}.

Finally, convergence (\ref{conv6}) follows for $\tau$ in a compact subset of $\R$ and Schwartz class matrix $W$ by a simple modification to the proof of (\ref{conv5}). The full statement of (\ref{conv6}) follows once again from operator norm bounds in Propositions \ref{prop1} and \ref{prop2} and Lemma \ref{uniformlem}.
\end{proof}

\subsection{Step 1: Parabolic Regularity and Compactness}

The main result of this section is the following proposition which is analogous to Proposition \ref{prop5} in the case of $\R^2 \times \T$.
\begin{proposition}\label{prop6}
Let $w_c(\tau)$ be a solution to (\ref{coremild}) on $\R^2 \times \R$ such that for some $-\infty < \tau^* \le \infty$, $m > 2$, and $\Lambda > 0$,
\begin{enumerate}[label=(\roman*)]
\item \label{assumptioni2} $\sup_{\tau < \tau^*}\norm{w_c(\tau)}_{B_z L^2(m)} \le \Lambda;$
\item \label{assumptionii2} $\lim_{\tau \to -\infty}\norm{w_c^\xi(\tau)}_{B_z L^2(m)} = 0$;
\item \label{assumptioniii2} $\sup_{\tau < \tau^*}\norm{\partial_z w_c^z(\tau)}_{B_z L^2(m)} < \infty$;
\item \label{assumptioniv2}  and $\sup_{\tau < \tau^*}\norm{z w_c^z(\tau)}_{B_z L^2(m)} < \infty$.
\end{enumerate}
Then,
\begin{equation}\label{bound1}
\sup_{\tau < \tau^*} \|\sdv w_c(\tau)\|_{B_zL^2(m)} \le C(\Lambda, m)
\end{equation}
and
\begin{equation}\label{bound2}
\lim_{\tau \rightarrow -\infty} \|\sdv w_c^\xi(\tau)\|_{B_zL^2(m)} = 0,
\end{equation}
and moreover, the trajectory $\{w_c(\tau)\}_{\tau < \tau^*}$ is precompact in $B_zL^2(m^\prime)$ for each $2 < m^\prime < m$.
\end{proposition}

As in the proof of Proposition \ref{prop5}, we prove this via a sequence of lemmas. Most of the arguments below will be nearly identical to those from Section 4.2, so we attempt to minimize repetition by emphasizing the differences and omitting identical arguments.

\begin{lemma}\label{lem18}
Let $w_c$ be a solution to (\ref{coremild}) such that for some $\tau^* > -\infty$, $m > 2$, and $\Lambda > 0$,
\begin{equation}
\sup_{\tau < \tau^*}\norm{w_c(\tau)}_{B_z L^2(m)} \le \Lambda.
\end{equation}
Then,
\begin{equation}
\sup_{\tau < \tau^*} \|\sdv_\tau w_c(\tau)\|_{B_zL^2(m)} \le C(\Lambda,m).
\end{equation}
\end{lemma}

\begin{proof}
This proved as in Lemma \ref{lem7}. In particular, we recall that $\sdv w_c$ solves the integral equation
\begin{equation}
\sdv_\tau w_c(\tau) = \sdv_\tau S_\alpha(\tau,\tau_0)w_c(\tau_0)  - \int_{\tau_0}^\tau \sdv_\tau S_\alpha(\tau,s)\left[ (v_c \cdot \sdv_s)w_c - (w_c \cdot \sdv_s)v_c\right](s) \ ds.
\end{equation}
Then, we estimate in $B_zL^2(m)$ and use the linear propagator estimates on $S_\alpha$ from Proposition \ref{prop2}, namely,
\begin{equation}\label{bound7}
\|\sdv_\tau S_\alpha(\tau,\tau_0)w_c\|_{B_zL^2(m)} \lesssim_{\alpha,m,\gamma,p} \frac{e^{\gamma(\tau - \tau_0)}}{a(\tau - \tau_0)^{\frac{1}{p}}} \|w_c\|_{B_zL^{p}(m)}
\end{equation}
for $p = 2$ or $p = 4/3$. Now, since (\ref{bound7}) does not use any particular vector structure, the proof proceeds identically to the proof of Lemma \ref{lem7} from here.
\end{proof}

\begin{lemma}\label{lem19}
Suppose $w_c(\tau)$ is a solution (\ref{coremild}) satisfying Assumptions \ref{assumptioni2}, \ref{assumptioniii2}, and \ref{assumptioniv2}. Then, for each $2 < m^\prime < m$, $\{w_c(\tau)\}_{\tau < \tau^*}$ is precompact in $B_zL^2(m^\prime)$.
\end{lemma}

\begin{proof}
Fix $2 < m^\prime < m$,  $\{\tau_n\}_{n=1}^\infty$ an arbitrary sequence of times, and let $\hat{w}_c^z(\tau,\zeta,\xi)$ denote the Fourier transform in $z$ of $w_c^z$, where $\zeta$ is the corresponding frequency variable. We note that $\|w_c^z(\tau)\|_{B_zL^2(m)} = \|\hat{w}_c^z(\tau)\|_{L_\zeta^1L_\xi^2(m)}$ and Assumption \ref{assumptioniii2} and Assumption \ref{assumptioniv2} imply
\begin{equation}\label{assumptions}
\|\zeta\hat{w}_c^z(\zeta)\|_{L_\zeta^1L_\xi^2(m^\prime)} \lesssim 1 \qquad\text{and}\qquad \|\partial_\zeta\hat{w}_c^z(\zeta)\|_{L_\zeta^1L_\xi^2(m^\prime)} \lesssim 1.
\end{equation}
Thus, we apply the Aubin-Lions-Simon Lemma (proved for the $L^1$ case by Simon in \cite{simon}) to the triple $$H^1(m) \embeds L^2(m^\prime) \embeds L^2(m^\prime)$$ to conclude that the embedding
$$\{L^1_\zeta([-Z,Z];H^1_\xi(m)) \ | \ \partial_\zeta \hat{w}_c^z \in L^1_\zeta([-Z,Z];L_\xi^2(m^\prime))\} \embeds L^1_\zeta([-Z,Z];L^2_\xi(m^\prime))$$
is compact for each $m>2$, $Z\in \N$. Because $\hat{w}_c^z$ satisfies the bounds in (\ref{assumptions}), we may use this compact embedding in conjunction with a diagonalization argument to obtain a subsequence $\{\tau_n\}$ and limit $\hat{w}$ such that $\{\hat{w}_c^z(\tau_n)\}$ converges to $\hat{w}$ in $L^1_{loc,\zeta}L^2_\xi(m^\prime)$. Again using the (tightness) bound in(\ref{assumptions}), we conclude that the convergence is in $L^1_\zeta L^2_\xi(m^\prime)$, or equivalently, $w_c^z(\tau_n) \rightarrow w$ in $B_zL^2(m^\prime)$.
\end{proof}

\begin{remark}
While the proof of the parabolic regularity in Lemma \ref{lem18} is almost identical to that in Lemma \ref{lem6}, the statement that $w^\xi_c \rightarrow 0$ implies $\sdv w^\xi_c \rightarrow 0$ requires some changes. Note that $S^\alpha$ mixes $w^\xi$ and $w^z$ in an a priori complicated way. As discussed at the beginning of Section 5.2, the purpose of the semigroup asymptotics in Lemmas \ref{lem17} and \ref{lem18} is exactly to replace $S_\alpha$ with $(\Gamma_\alpha,\Tau_\alpha)$, which is block diagonal.
\end{remark}

We begin to prove the $\sdv w^\xi_c \rightarrow 0$ by isolating the following compactness argument which we will need again.
\begin{lemma}\label{lem20}
Let $\{w_c\}_{\tau < \tau^*}$ be a precompact set in $B_zL^2(m)^3$ for some $m >2$. Then,
\begin{equation}
U = \{v_c(s) \tens w_c(s) - w_c(s) \tens v_c(s)\}_{s < \tau^*}
\end{equation}
is precompact as a subset of $B_zL^{4/3}(m)^9$.
\end{lemma}

\begin{proof}
First, we define a bilinear form for each $s$, $\F_s: B_zL^2(m)^3\oplus B_zL^2(m)^3 \rightarrow B_zL^{4/3}(m)^9$ by
\begin{equation}
\F_s(f,g) = \sdv_s \times (\slap_s)^{-1}f \tens g - g\tens \sdv_s \times (\slap_s)^{-1}f.
\end{equation}
By definition, we see that for $K = \{w_c(\tau)\}_{\tau < \tau^*}$,
$$U \subset \bigcup_{s < \tau^*} \F_s(K,K).$$
Therefore, it suffices to show that the larger set on the right hand side is precompact.

Second, we note that by H\"older's inequality, the Biot-Savart law (\ref{biotsavart}), and the embedding, \linebreak $B_zL^{4/3}\embeds B_zL^2(m)$, we have the estimate,
\begin{equation}
\begin{aligned}
\|\F_s(f_1 - f_2,g_1 - g_2)\|_{B_zL^{4/3}(m)} \lesssim_m& \|f_1 - f_2\|_{B_zL^2(m)}\left(\|g_1\|_{B_zL^2(m)} + \|g_2\|_{B_zL^2(m)}\right)\\
	& + \|g_1 - g_2\|_{B_zL^2(m)}\left(\|f_1\|_{B_zL^2(m)} + \|f_2\|_{B_zL^2(m)}\right),
\end{aligned}
\end{equation}
where the implicit constant is independent of $s$. This implies that restricting $\F_s$ to the compact set $\overline{K}\times \overline{K}$ gives a family of Lipschitz functions with Lipschitz constant uniformly bounded in $s$.

Third, we will show that the compactness of $\overline{K}$ implies that the set
$$\bigcup_{s < \tau^*} \F_s(K,K)$$
is totally bounded in $B_zL^{4/3}(m)^9$ and hence precompact.
Fix $\epsilon > 0$ arbitrary. Now, by Arzela-Ascoli, $\{\F_s\}_{s < \tau^*}$ is precompact as a subset of $C(\overline{K}\times \overline{K};B_zL^{4/3}(m)^9)$.
So, there is a collection of $N_\epsilon$ balls $B_{\epsilon/2}(c_i) \subset C(\overline{K}\times \overline{K};B_zL^{4/3}(m)^3)$ such that $\{B_{\epsilon/2}(c_i)\}$ covers $\{\F_s\}_{s < \tau^*}$.
Look at any $c_i$ and note $c_i(K,K)$ is precompact as a subset of $B_zL^{4/3}(m)^9$ since $c_i$ is continuous.
Therefore, there are $N_\epsilon^i$ open balls $B_{\epsilon/2}(d_j^i) \subset B_zL^{4/3}(m)^3$ covering $c_i(K,K)$.

Note, since $e \in B_{\epsilon/2}(c_i)$ implies for each $h\in K$,
$\|e(h,h) - c_i(h,h)\|_{B_zL^{4/3}(m)} < \epsilon/2$
and $c_i(h,h) \in B_{\epsilon/2}(d_j^i)$ for some $1\le j \le N_\epsilon^i$, $e(h,h) \in B_{\epsilon}(d_j^i)$.

Thus, $U$ is covered by the finite collection $\{B_\epsilon(d_j^i)\}_{i,j}$ of balls of radius $\epsilon$. Since $\epsilon > 0$ was arbitrary, $U$ is precompact.
\end{proof}

\begin{lemma}\label{lem22}
Let $w_c$ be a mild solution to (\ref{coremild}) such that for some $m > 2$, $\tau^* > -\infty$, and $\Lambda > 0$,
\begin{equation}
\sup_{\tau < \tau^*} \|w_c\|_{B_zL^2(m)} \le \Lambda.
\end{equation}
Moreover, suppose that $w^\xi_c \rightarrow 0$ in $B_zL^2(m)$ and $\{w_c^z\}_{\tau < \tau^*}$ is compact in $B_zL^2(m)$. Then, $\sdv w^\xi_c \rightarrow 0$ in $B_zL^2(m)$ and, furthermore, $\sdz w^z_c \rightarrow 0$ in $B_zL^2(m)$.
\end{lemma}

\begin{proof}
For clarity, we prove the $\nabla_\xi w_c^\xi \rightarrow 0$ and $\sdz w_c^\xi \rightarrow 0$ separately. We begin by recalling that $\nabla_\xi w_c^\xi$ satisfies the integral equation
\begin{equation}
\nabla_\xi w_c^\xi(\tau) = \nabla_\xi S^\xi_\alpha(\tau,\tau_0)w_c(\tau_0) - \int_{\tau_0}^{\tau} \nabla_\xi S_\alpha^\xi(\tau,s)\divz_s \left(w_c\tens v_c - v_c \tens w_c\right) \ ds,
\end{equation}
for each $\tau_0  < \tau < \tau^*$.
Now, we replace $S^\xi_\alpha$ with $\Gamma_\alpha$ to obtain
\begin{equation}
\begin{aligned}
\nabla_\xi w_c^\xi(\tau) &= \left(\nabla_\xi S^\xi_\alpha(\tau,\tau_0)w_c(\tau_0) - \nabla_\xi \Gamma_\alpha(\tau -\tau_0)w_c^\xi(\tau_0)\right) +  \nabla_\xi \Gamma_\alpha(\tau - \tau_0)w_c^\xi(\tau_0)\\
	& - \int_{\tau_0}^{\tau} \nabla_\xi S_\alpha^\xi(\tau,s)\divz_s \left(w_c\tens v_c - v_c \tens w_c\right) - \nabla_\xi \Gamma_\alpha(\tau,s)\div_\xi \left(w_c^{\xi}\tens v_c^{\xi} - v_c^{\xi} \tens w_c^{\xi}\right) \ ds\\
	& - \int_{\tau_0}^\tau \nabla_\xi \Gamma_\alpha(\tau,s)\div_\xi \left(w_c^{\xi}\tens v_c^{\xi} - v_c^{\xi} \tens w_c^{\xi}\right) \ ds\\
	&=: I_1(\tau, \tau_0) - \int_{\tau_0}^\tau I_2(\tau,s)\ ds\\
	&\qquad+ \nabla_\xi \Gamma_\alpha(\tau - \tau_0)w_c^\xi(\tau_0) - \int_{\tau_0}^\tau \nabla_\xi \Gamma_\alpha(\tau,s)\div_\xi \left(w_c^{\xi}\tens v_c^{\xi} - v_c^{\xi} \tens w_c^{\xi}\right) \ ds,
\end{aligned}
\end{equation}
where we obtain two error terms $I_1$ and $I_2$ coming from the difference in linear propagators. We control both $I_1$ and $I_2$ in $B_zL^2(m)$ using Lemma \ref{lem16} and Lemma \ref{lem17}.
First, we use Lemma \ref{lem16}
and the compactness of $\{w_c\}_{\tau < \tau^*}$ to conclude that additional error term satisfies for $T > 0$,
$$\lim_{\tau_0\rightarrow -\infty} \sup_{\tau_0 < \tau < \tau_0 + T} \|I_1(\tau,\tau_0)\|_{B_zL^2(m)} = 0.$$
Second, we use Lemma \ref{lem20} to conclude that the nonlinearity $\left\{w_c(s) \tens v_c(s) - v_c(s)\tens w_c(s)\right\}_{s<\tau^*}$ is compact in $B_zL^{4/3}(m)$ and, therefore, by Lemma \ref{lem17},
$$\lim_{\tau_0\rightarrow -\infty} \sup_{\tau_0 < s < \tau < \tau_0 + T} \|I_2(\tau,s)\|_{B_zL^2(m)} = 0,$$ for $T>0$. This implies that the sequence of real-valued functions
$$(s,\tau)\mapsto \|\chi_{[-n,-n+T]}(s)\chi_{[-n,-n+T]}(\tau)I_2(\tau,s)\|_{B_zL^2(m)}$$
converge uniformly to $0$, where $\chi_A$ denotes the indicator function of $A$. Thus, the integral error term converges uniformly to $0$ and we conclude that $\nabla_\xi w^\xi_c$ satisfies
\begin{equation}
\nabla_\xi w_c^\xi(\tau) =\mathcal{E}(\tau,\tau_0) + \nabla_\xi \Gamma_\alpha(\tau - \tau_0)w_c^\xi(\tau_0) - \int_{\tau_0}^\tau \nabla_\xi \Gamma_\alpha(\tau,s)\div_\xi \left(w_c^{\xi} \tens v_c^{\xi} - v_c^{\xi}\tens w_c^{\xi}\right) \ ds,
\end{equation}
where
$$\lim_{\tau_0\rightarrow -\infty}\sup_{\tau_0 < \tau < \tau_0 + T} \mathcal{E}(\tau,\tau_0) = 0,$$
for any $T > 0$. From here, the proof proceeds as in the proof of Lemma \ref{lem6} with minor modifications.
Similarly, $\sdz w_c^\xi$ satisfies
\begin{equation}
\sdz w_c^\xi(\tau) = \sdz S^\xi_\alpha(\tau,\tau_0)w_c(\tau_0) - \int_{\tau_0}^{\tau} \sdz S_\alpha^\xi(\tau,s)\divz_s \left(w_c \tens v_c - v_c\tens w_c\right) \ ds.
\end{equation}
Now, since $\{w_c(\tau)\}_{\tau < s}$ is compact, a direct application of the convergence statement (\ref{conv4}) in Lemma \ref{lem17}, yields
\begin{equation}
\lim_{\tau_0\rightarrow -\infty} \sup_{\tau_0 < \tau < \tau_0 + T}\|\sdz S^\xi_\alpha(\tau,\tau_0)w_c(\tau_0)\|_{B_zL^2(m)} = 0,
\end{equation}
for any $T > 0$. Moreover, since $\{w_c \tens v_c - v_c\tens w_c\}_{s <\tau^*}$ is compact in $B_zL^{4/3}(m)$ by Lemma \ref{lem17}, a direct application of convergence (\ref{conv6}) in Lemma \ref{lem17} yields
\begin{equation}
\lim_{\tau_0\rightarrow -\infty} \sup_{\tau_0 < s < \tau < \tau_0 + T}\|\sdz S_\alpha^\xi(\tau,s)\divz_s \left(w_c \tens v_c - v_c\tens w_c\right)(s)\|_{B_zL^2(m)} = 0,
\end{equation}
for any $T>0$. This completes the proof that $\sdv_{\tau_0} w^\xi(\tau_0) \rightarrow 0$ in $B_zL^2(m)$. Finally, as in Lemma \ref{lem6}, the divergence-free condition implies that $\sdz w^z(\tau_0) \rightarrow 0$ in $B_zL^2(m)$.
\end{proof}

\subsection{Step 2: Error Estimates}

We now estimate the error terms $R^\prime(\tau)$, which we recall are
\begin{equation}\label{R2}
R^\prime = \div_\xi(v^\xi_c w_c^z - w_c^\xi v^z_c) - \nabla_\xi^\perp(\slap)^{-1}w_c^z \cdot \nabla_\xi w_c^z.
\end{equation}

\begin{lemma}\label{lem21}
Under the assumptions of Proposition \ref{prop6} with $m>2$, the error terms $R^\prime(\tau)$ defined in (\ref{R2}), satisfy
\begin{equation}
\lim_{\tau\rightarrow -\infty}\|R^\prime(\tau)\|_{B_zL^{4/3}(m)} = 0.
\end{equation}
\end{lemma}

\begin{proof}
Expanding $R^\prime$ using the product rule, $\div_\xi v^\xi_c = -\sdz v^z_c$ and $\div_\xi w^\xi = - \sdz w^z$, we obtain
\begin{equation}
R^\prime(\tau) = v_c \cdot \sdv w^z_c - w_c \cdot \sdv v^z_c - \nabla_\xi^\perp(\slap)^{-1}w_c^z \cdot \nabla_\xi w_c^z,
\end{equation}
which is estimated as in Lemma \ref{lem9}.

\end{proof}

\subsection{Step 3: Invariance of $\alpha$-limit set}

Define $\Phi^\prime(\tau)$ to be the flow on $B_zL^2(m)$ generated by the 2d vorticity equation in the form
\begin{equation}\label{Tflow}
w^*(\tau) = \Tau_\alpha(\tau - \tau_0)w^*(\tau_0) - \int_{\tau_0}^\tau \Tau_\alpha(\tau - s)[\nabla_\xi^\perp \Delta_\xi^{-1}w^*(s) \cdot \nabla_\xi w^*(s)]\ ds.
\end{equation}
Moreover, let $\A^\prime$ be the $\alpha$-limit set of $\{w_c^z(\tau)\}_{\tau < \tau^*}$ in $B_zL^2(m)$. By analogy to Lemma \ref{lem11}, we will show that $\A^\prime$ is invariant under $\Phi^\prime(\tau)$.
\begin{remark}
Since $\Tau_\alpha(\tau - \tau_0)$ satisfies the same estimates as $e^{(\tau -\tau_0)\L}$, the proof of Lemma \ref{flowlem} also yields $\Phi^\prime$ exists for a uniform amount of time on compact subsets of $B_zL^2(m)$ and extends $\varphi^\prime$, the flow generated by (\ref{Tflow}) on $L^2(m) \subset L^2(\R^2)$.\end{remark}
\begin{remark}
Note that like $\Phi$, $\Phi^\prime$ is scalar-valued and is only meant to describe the asymptotic evolution of the $z$-component of the vorticity.
\end{remark}

\begin{lemma}\label{lem13}
Suppose $w(\tau)$ is a mild solution to (\ref{coremild}) satisfying the hypotheses of Proposition \ref{prop6} for some $m > 2$, $\Lambda > 0$, and $\tau^* > -\infty$. Then, for each $2 < m^\prime < m$, the $\alpha$-limit set $\A^\prime \subset B_zL^2(m^\prime)$ of $\{w^z_c(\tau)\}_{\tau < \tau^*}$, $\A^\prime$ is nonempty and there exists a $T(\A^\prime)$ such that $\Phi^\prime(\tau)\A^\prime = \A^\prime$ for each $0\le \tau \le T(\A^\prime)$.
\end{lemma}

\begin{proof}

The proof mostly follows that of Lemma \ref{lem11} with a few changes. As in Lemma \ref{lem11}, fix $2 < m^\prime < m$ and note that by the compactness statement from Proposition \ref{prop6}, we have already shown $\A^\prime \subset B_zL^2(m^\prime)$ is nonempty and compact. Then, we take $w^\infty \in \A^\prime$ and $\tau_n \rightarrow -\infty$ such that $w_c^z(\tau_n) \rightarrow w^\infty$ in $B_zL^2(m^\prime)$ and let $\tau > 0$.

We recall that the functions $w_c^z(\tau+\tau_n)$ satisfy the following integral equations with initial data at time $t_n$:
\begin{equation}\label{eqn6}
w_c^z(\tau + \tau_n) = S_\alpha^z(\tau+\tau_n,\tau_n)w_c(\tau_n) - \int_{\tau_n}^{\tau+\tau_n} S_\alpha^z(\tau+\tau_n,s)\divz_s\left[w_c\tens v_c - v_c \tens w_c\right](s) \ ds.
\end{equation}
Also, recall $w^\infty(\tau)$ satisfies the integral form of (\ref{svorticity2}), which we write as
\begin{equation}
w^\infty(\tau) = \Tau_\alpha(\tau)w^\infty - \int_0^\tau \Tau_\alpha(\tau-s)\left[\nabla_\xi^\perp\Delta_\xi^{-1}w^\infty \cdot \nabla_\xi w^\infty\right](s) \ ds.
\end{equation}
Changing variables in (\ref{eqn6}),
we expand the difference between the two equations as
\begin{equation}
\begin{aligned}
w^z(\tau + \tau_n) -& w^\infty(\tau) = S_\alpha^z(\tau + \tau_n,\tau_n)w_c(\tau_n) - \Tau(\tau)w^\infty\\
	&- \int_0^{\tau} S_\alpha^z(\tau+\tau_n,s + \tau_n)\divz_s \left(w_c\tens v_c - v_c\tens w_c\right)(s + \tau_n)\\
	&\qquad\qquad - \Tau_\alpha(\tau - s)\div_\xi \left(w_c^z \tens v_c^\xi - v_c^z \tens w_c^\xi\right) \ ds\\
	&- \int_0^\tau \Tau_\alpha(\tau - s)\div_\xi \left[(v_c^\xi \tens w_c^z - w_c^\xi \tens v_c^z) - w_c^z\nabla_\xi^\perp\Delta_\xi^{-1}w_c^z\right](s + \tau_n) \ ds\\
	&+ \int_0^\tau \Tau_\alpha(\tau -s)\left[\nabla_\xi \cdot (w^\infty \nabla_\xi^\perp\Delta_\xi^{-1}w^\infty(s))\right]\\
	&\qquad\qquad- \Tau_\alpha(\tau-s)\left[\nabla_\xi \cdot (w_c^z \nabla_\xi^\perp\Delta_\xi^{-1}w_c^z(s + \tau_n))\right] \ ds\\
	&=: L_n(\tau) + E^2_n(\tau) + E^1_n(\tau) + I_n(\tau).
\end{aligned}
\end{equation}
We bound each term in $B_zL^2(m^\prime)$ individually. The first term $L_n$ is handled as in Lemma \ref{lem11} replacing (\ref{conv1}) with (\ref{conv3}). The second term $E_n^2$ is handled as in Lemma \ref{lem11} replacing (\ref{conv2}) with (\ref{conv5}) and using Lemma \ref{lem20}. The third term $E_n^1$ contains the $R^\prime$ part of the error and is handled as in Lemma \ref{lem11} using Lemma \ref{lem21} in place of Lemma \ref{lem9}. Finally, the last term $I_n$ is handled exactly as in Lemma \ref{lem11}.
\end{proof}

\subsection{Step 4: Rigidity}

We now show that the flows $\Phi(\tau)$ and $\alpha G + \Phi^\prime(\tau)$ have the same invariant sets, where $\Phi$ is the flow generated by the standard mild formulation of (\ref{svorticity2}), namely
\begin{equation}
w^*(\tau) = e^{\tau\L}w^*_0 - \int_{0}^\tau e^{(\tau - s)\L}[\nabla^\perp_\xi\Delta^{-1}_\xi w^*(s) \cdot \nabla_\xi w^*(s)] \ ds.\label{mild2d}
\end{equation}

\begin{lemma}
For each $m > 2$, $f\in B_zL^2(m)$, and $\tau \ge 0$, $\Phi(\tau)f = \alpha G + \Phi^\prime(\tau)f$.
\end{lemma}

\begin{proof}
Since $\Phi(\tau)$ and $\Phi^\prime(\tau)$ are extensions of the corresponding two-dimensional flows $\phi(\tau)$ and $\phi^\prime(\tau)$ on $L^2(m) \subset L^2(\R^2)$, it suffices to show that $\phi(\tau)f = \alpha G +  \phi^\prime(\tau)f$ for each $f\in L^2(m)$. Thus, take $w_1(\tau) = \alpha G + \phi^\prime(\tau)f$ and $w_2(\tau) = \phi(\tau)f$ so that $w_1,w_2 \in C([0,\tau],L^2(m))$ such that $w_1(0) = w_2(0)$ for $w_1$ solving (\ref{mildlimit}) and $w_2$ solving (\ref{mild2d}). Then, certainly, by parabolic regularity, $w_1(\tau), w_2(\tau)$ are smooth for $\tau > 0$. Thus, $w_1(\tau)$ solves (\ref{mild2d}) when the initial data is taken at time $\tau_0 > 0$. That is,
\begin{equation}
w_1(\tau) = e^{(\tau - \tau_0)\L}w_1(\tau_0) - \int_{\tau_0}^\tau e^{(\tau-s)\L}\left[\left(\nabla^\perp \Delta^{-1} w_1(s) \right) \cdot \nabla w_1(s)\right] \ ds.
\end{equation}
Taking the limit as $\tau_0 \rightarrow 0$ and using strong continuity in $L^2(m)$, we find that $w_1$ satisfies (\ref{mild2d}) with initial data taken at time $\tau_0 = 0$. Hence, the uniqueness of solutions to (\ref{mild2d}) guarantees $w_1 = w_2$ in $C([0,\tau];L^2(m))$ and completes the proof.
\end{proof}

\begin{proof}[Proof of Theorem \ref{thm2}]
By Lemma \ref{mildlemma}, $w_c$ satisfies the integral equation (\ref{coremild}). Moreover, by assumption, $w_c$ satisfies the hypotheses of Proposition \ref{prop6} for some $\tau^* > -\infty$, $\Lambda > 0$, and $m > 2$. Taking $2 < m^\prime < m$, we have that $\A^\prime$, the $\alpha$-limit set of $\{w^z_c(\tau)\}_{\tau < \tau^*}$ is precompact in $B_zL^2(m^\prime)$, and, by the preceding lemma, $\A^\prime$ is invariant under the flow $\Phi(\tau)$ generated by the 2d vorticity equation. Thus, by the argument in Section 4.5, we conclude $\A^\prime = \{0\}$. Equivalently, by the compactness of $\A^\prime$,
$$\lim_{\tau \rightarrow -\infty} \left\|w - \begin{bmatrix} 0 \\ 0 \\ \alpha G\end{bmatrix}\right\|_{B_zL^2(m^\prime)} = 0.$$
Hence, $\omega$ lies within the uniqueness class shown in \cite{jacob} and the proof is complete. 
\end{proof}

\medskip
 

\end{document}